\documentclass[11pt,a4paper]{article}

\usepackage[english]{babel}
\usepackage{amsmath}
\usepackage{amssymb}
\usepackage{amsfonts}
\usepackage{array}
\usepackage[dvips]{graphicx}
\usepackage{epsfig}
\usepackage{float}
\usepackage{algorithm}
\usepackage{algorithmic}
\usepackage{mathdots}
\usepackage{bm}
\usepackage{enumitem}  

\newcommand{\ds}[1]{\displaystyle{#1}}

\newcommand{\Lspace}[2]{L^{#1}\left(\Omega;#2\right)}

\newcommand{\Cspace}[3]{\mcC^{#1}\left(#2;#3\right)}
\newcommand{\Cspacepedex}[4]{\mcC^{#1}_{#2}\left(#3;#4\right)}
\newcommand{\Cspacemix}[3]{\mcC^{#1,mix}\left(#2;#3\right)}
\newcommand{\Cspacemixpedex}[4]{\mcC^{#1,mix}_{#2}\left(#3;#4\right)}

\newcommand{\Holder}[2]{\mathcal{C}^{#1}\left(#2\right)}
\newcommand{\Holdermix}[2]{\mathcal{C}^{#1,mix}\left(#2\right)}

\newcommand{\Id}{\mathrm{Id}}
\newcommand{\e}{\textmd{e}}

\newcommand{\ddx}{\text{d}x}

\newcommand{\Div}{\textrm{div}}

\newcommand{\Grad}{\nabla}

\newcommand{\Trdiag}[1]{\mathrm{Tr}_{|_{1:#1}}}
\newcommand{\Trdiaggen}[2]{\mathrm{Tr}_{|_{#1:#2}}}

\newcommand{\R}{\mathbb R}
\newcommand{\N}{\mathbb N}

\newcommand{\T}{\mathcal T}

\newcommand{\mcC}{\mathcal C}
\newcommand{\mcL}{\mathcal L}

\newcommand{\vet}[1]{\mathbf{#1}}

\newcommand{\bc}[2]{\left(\begin{array}{c}#1\\#2\end{array}\right)}

\newcommand{\mean}[1]{\mathbb E\left[#1\right]}

\newcommand{\corr}[1]{E^{#1}}
\newcommand{\corrsd}[1]{E_h^{#1}}
\newcommand{\corrfd}[1]{E_{L,h}^{#1}}

\newcommand{\norm}[2]{\left\|#1\right\|_{#2}}
\newcommand{\seminorm}[2]{\left|#1\right|_{#2}}
\newcommand{\abs}[1]{\left|#1\right|}

\newcommand{\der}[1]{\partial^{{#1}}}
\newcommand{\hder}[2]{\partial^{#1}_{#2}}

\newcommand{\interp}[1]{P_{#1}}
\newcommand{\spinterp}[1]{\widehat P_{#1}}
\newcommand{\spfe}[1]{\widehat V_{#1}}

\newtheorem{theorem}{Theorem}
\newtheorem{proposition}[theorem]{Proposition}
\newtheorem{corollary}[theorem]{Corollary}
\newtheorem{lemma}[theorem]{Lemma}
\newtheorem{definition}[theorem]{Definition}
\newtheorem{remark}[theorem]{Remark}

\usepackage[normalem]{ulem}
\usepackage[bookmarks=true,pdftex=false,bookmarksopen=true,pdfborder={0,0,0}]{hyperref}

\newenvironment{proof}{\noindent{\bf Proof. \/}\begin{small}\noindent}{\hfill\EndProofMarker\end{small}}
\newcommand{\EndProofMarker}{$\Box$}

\begin{document}

\title{
Regularity and sparse approximation of the recursive first moment equations for the lognormal Darcy problem
\footnote{F. Bonizzoni acknowledges partial support from the Austrian Science Fund (FWF) through the project F 65, and has
been supported by the FWF Firnberg-Program, grant T998.}
}

\author{Francesca Bonizzoni$^\sharp$, Fabio Nobile$^\dag$}

\maketitle 

\begin{center}
{\small 
$^\sharp$ Faculty of Mathematics\\
				University of Vienna\\ 
				Oskar-Morgenstern-Platz 1,1090 Wien, Austria\\
{\tt francesca.bonizzoni@univie.ac.at}\\
$^\dag$ CSQI -- MATHICSE\\
	Ecole Polytechnique F\'ed\'erale de Lausanne\\
	Station 8, CH-1015 Lausanne, Switzerland\\
{\tt fabio.nobile@epfl.ch}\\	
}
\end{center}

\date{}


\vspace*{0.5cm}


\vspace*{0.5cm}

\begin{abstract}
We study the Darcy boundary value problem with log-normal permeability field. We adopt a perturbation approach, expanding the solution in Taylor series around the nominal value of the coefficient, and approximating the expected value of the stochastic solution of the PDE by the expected value of its Taylor polynomial. The recursive deterministic equation satisfied by the expected value of the Taylor polynomial (first moment equation) is formally derived. Well-posedness and regularity results for the recursion are proved to hold in Sobolev space-valued H\"older spaces with mixed regularity. The recursive first moment equation is then discretized by means of a sparse approximation technique, and the convergence rates are derived.
\end{abstract}

%%%%%%%%%%%%%%%%%%%%%%%%%%%%%%%%%%%%%%%%%%%%%%%%%%%%%%%%%%%%%%%%%%%%%%
\section{Introduction}
\label{Introduction}
%%%%%%%%%%%%%%%%%%%%%%%%%%%%%%%%%%%%%%%%%%%%%%%%%%%%%%%%%%%%%%%%%%%%%%

In many applications, the input parameters of the mathematical model describing the system behavior are unavoidably affected by uncertainty, as a consequence of the incomplete knowledge or the intrinsic variability of certain phenomena. 
\emph{Uncertainty Quantification} (UQ) conveniently incorporates the input variability or lack of knowledge inside the model, often by describing the uncertain parameters as random variables or random fields, and aims to infer the uncertainty in the solution of the model, or the specific output quantities of interest, by computing their statical moments.

The physical phenomenon we are interested in this work is the single-phase flow of a fluid in a bounded heterogenous saturated porous medium. In particular, we consider the following stochastic partial differential equation (PDE), named the Darcy problem, posed in the complete probability space $(\Omega,\mathcal F,\mathbb P)$ and in the bounded physical domain $D\subset\R^d$ ($d=2,3$):
\begin{equation}
\label{eq:darcy_strong}
	- \Div\left(e^{Y(\omega,x)} \Grad u(\omega,x)\right)=f(x)
	\quad \text{for }x\in D\textmd{ and a.e. }\omega\in\Omega
\end{equation}
endowed with suitable boundary conditions on $\partial\Omega$, where $u(\omega,x)$ represents the hydraulic head, the forcing term $f(x)\in L^2(D)$ is deterministic, and the permeability coefficient $\e^{Y(\omega,x)}$ is modeled as a \emph{log-normal random field}, $Y(\omega,x)$ being a centered Gaussian random field with small standard deviation. The log-normal diffusion problem~\eqref{eq:darcy_strong} is widely used in geophysical applications (see, e.g.,~\cite{Bear2010,Dagan2012,Xia2019,Li2003} and the references there), and has been studied mathematically, e.g., in~\cite{Charrier2015,Bonizzoni2014,Charrier2013,Gittelson2010}.

Under suitable assumptions on the covariance of the random field $Y(\omega,x)$, it is possible to show that the Darcy problem is well-posed (see~\cite{Charrier2013}). 

Given complete statistical information on the Gaussian random field $Y(\omega,x)$, and assuming that each realization $Y(\omega,\cdot)$ is almost surely H\"older continuous with parameter $\gamma$, the aim of the present work is to construct an approximation for the expected value of the stochastic solution $\mean{u}$. To this end, we adopt a \emph{perturbation approach}, in which the stochastic solution $u$ is viewed as the map $u:\Holder{0,\gamma}{\bar D}\rightarrow H^1(D)$ which associates to each realization $Y(\omega,\cdot)\in \Holder{0,\gamma}{\bar D}$, the unique solution $u(\omega,\cdot)$ of~\eqref{eq:darcy}, and is expanded in Taylor series w.r.t. $Y$, i.e., $\sum_{k=0}^{+\infty}\frac{u^k(Y,x)}{k!}$, $u^k$ being the $k$-th Gateaux derivative of $u$ w.r.t. $Y$.
The expected value of $u$ is then approximated as 
$$
\mean{u}(x)\simeq\mean{T^Ku}(x)=\sum_{k=0}^K\frac{\mean{u^k}(x)}{k!},
$$
where $T^Ku(Y,x)$ denotes the $K$-th degree Taylor polynomial. We refer to $\mean{u^k}$ as the $k$-th order correction to the expected value of $u$, and to $\mean{T^K u}$ as the $K$-th degree approximation of the expected value of $u$. 

In~\cite{Bonizzoni2012,Bonizzoni2013,Bonizzoni2014} the authors show that, as $K$ goes to infinity, the $K$-th order approximation of the expected value of $u$ may actually diverge, for any positive value of the standard deviation $\sigma:=\sqrt{\frac{1}{\abs{D}}\int_D \mean{Y^2}(x)dx}$ of the random field $Y(\omega,x)$. Nevertheless, for $\sigma$ and $K$ small enough, $\mean{T^K u}$ provides a good approximation of $\mean{u}$. The work~\cite{Bonizzoni2014} also provides an estimate of the optimal degree of the Taylor polynomial achieving minimal error, for any given $\sigma>0$.

If a finite-dimensional approximation of the random field $Y(\omega,x)$ via $N$ random variables is available (e.g., by using the Karhunen-Lo\'eve (KL) expansion), then the (multi-variate) Taylor polynomial can be explicitly computed (see, e.g., the geophysical literature~\cite{Zhang2004,Lu2004,Lu2007,Liu2007}). However, this approach entails the computation of $\bc{N+K}{K}$ derivates. To alleviate the curse of dimensionality, adaptive algorithms have been proposed in~\cite{Chkifa2013,Cohen2015} for the case of uniform random variables.

In the present paper we consider the entire field $Y(\omega,x)$, and not a finite dimensional approximation of it, hence the Taylor polynomial can not be directly computed. Following~\cite{Guadagnini1999,Guadagnini1999a,Todor2005}, we adopt the \emph{moment equations} approach, that is, we solve the deterministic equations satisfied by $\mean{u^k}$, for $k\geq 0$. 

In~\cite{Bonizzoni2016} the authors derive analytically the recursive problem solved by $\mean{u^k}$, which requires the recursive computation of the $(i+1)$-points correlations \\$\mean{u^{k-i}\otimes Y^{\otimes i}}$, with $i=k, k-1,\ldots,1$. These functions being high dimensional, a full tensor product finite element discretization is impractical and suffer the curse of dimensionality. To overcome this issue, in~\cite{Bonizzoni2016} the authors have proposed a low rank approximation of the fully (tensor product) discrete problem, using the Tensor Train format. The effectiveness of the method is shown with both one and two-dimensional numerical examples.

The present paper complements the above-mentioned results. The main achievement consists in the well-posedness and regularity results for the recursive first moment equation. These results are developed in the framework of $p$-integrable Lebesgue spaces. In particular, the key tool consists in showing that the diagonal trace of functions in the $L^p(D)$ space-valued mixed $\gamma$-H\"older space, belongs to $L^p(D)$, whenever $p>\frac{2d}{\gamma}$. 
We also address the discretization of the moment equations. Differently from~\cite{Bonizzoni2016}, to alleviate the curse of dimensionality we propose here a sparse approximation method based on the Smolyak construction, which is more amenable to error analysis. We present then a complete convergence analysis of the proposed discretization method.

The paper is organized as follows: in Section~\ref{sec:analytical_derivation_Eu} we recall the recursion solved by the $k$-th order correction $\mean{u^k}$, under the assumption that every quantity is well-defined, and every problem is well-posed. In Section~\ref{sec:Preliminaries}, we first introduce the Banach space-valued maps with mixed H\"older regularity, and then study the H\"older regularity of the diagonal trace of Sobolev space-valued mixed H\"older maps. These technical results will be needed in Section~\ref{sec:Recursive equation for Eu} to study the well-posedness and regularity of the recursion for $\mean{u^k}$. Section~\ref{sec:Sparse discretization} is dedicated to the sparse discretization of the recursion and its error analysis. Finally, we draw some conclusions in Section~\ref{sec:conclusions}.

%%%%%%%%%%%%%%%%%%%%%%%%%%%%%%%%%%%%%%%%%%%%%%%%%%%%%%%%%%%%%%%%%%%%%%
\section{Analytical derivation of the first moment equation}
\label{sec:analytical_derivation_Eu} 
%%%%%%%%%%%%%%%%%%%%%%%%%%%%%%%%%%%%%%%%%%%%%%%%%%%%%%%%%%%%%%%%%%%%%%

The weak formulation of the Darcy PDE~\eqref{eq:darcy_strong} endowed with homogeneous Dirichlet boundary conditions reads:
\begin{equation}
\label{eq:darcy}
\int_D e^{Y(\omega,x)} \Grad u(\omega,x) \cdot\Grad v(x) dx
	= \int_D f(x) v(x) dx,
	\quad\forall v\in H^1_{0}(D), \textmd{ a.s. in }\Omega.
\end{equation}
We assume here that the random field $Y\in L^s(\Omega,\Holder{0,\gamma}{\bar D})$ ($0<\gamma<1/2$) for all $1\leq s<+\infty$. Then, for any $f\in L^p(D)$, $1<p<+\infty$, the boundary value problem~\eqref{eq:darcy} admits a unique solution $u\in\Lspace{p}{H^1(D)}$, which depends continuously on the data (see~\cite{Charrier2013}). In particular, the H\"older regularity assumption $Y\in L^s(\Omega,\Holder{0,\gamma}{\bar D})$ ($0<\gamma<1/2$) for all $1\leq s<+\infty$, is fulfilled of the covariance function $Cov_Y\in\Holder{0,t}{\overline{D\times D}}$ for some $2\gamma<t\leq 1$ (see~\cite{Bonizzoni2013,Bonizzoni2014}).
The mentioned well-posedness result extends to the case of
uniform/non-uniform Neumann as well as mixed Dirichlet-Neumann boundary conditions. In particular, the limit situation of Neumann boundary conditions on $\partial D$ leads to the uniqueness of the solution $u(\omega,x)$ up to a constant. For clarity of presentation, in this work we restrict to the case of homogeneous Dirichlet boundary conditions in the rest of the paper.

In this section we recall (see~\cite{Bonizzoni2013,Bonizzoni2016}) the structure of the problem solved by $\mean{u^k}$ - the $k$-th order correction of the expected value of $u$ - assuming that every quantity is well-defined and every problem is well-posed. We will detail these theoretical aspects in the next sections.
 
Let $D\subset\R^d$, be such that $\partial D\in C^1$.
Let $p,q$ be real numbers such that $1<p,q<\infty$, with $\frac{1}{p}+\frac{1}{q}$=1 and $p>\frac{2d}{\gamma}$, where $\gamma$ is the H\"older regularity of the random field $Y$. 
The requirement $p>\frac{2d}{\gamma}$ will be clarified later (see Proposition~\ref{prop:traccia_Wp}).
Given $f\in L^p(D)$, $1<p<+\infty$, we define the linear form $\mathcal F\in (W^{1,q}_0(D))^\star$ as
\begin{equation*}
    \mathcal F(v):=\int_D f v\ dx,\quad\forall\, v\in W^{1,q}_0(D).
\end{equation*}
The correction of order 0, $u^0$, is deterministic and is the \emph{unique weak solution} of the following problem: find $u^0\in W^{1,p}_0(D)$ such that 
\begin{equation}
    \label{eq:u0}
    \int_D\nabla u^0\cdot\nabla v\ dx=\mathcal F(v) \quad\forall v\in W^{1,q}_0(D),
\end{equation} 
where $\frac{1}{p}+\frac{1}{q}=1$. Moreover, it exists $C=C(D)>0$ such that
\begin{equation}
    \label{eq:u0_bound}
    \norm{u^0}{W^{1,p}(D)}\leq C \norm{f}{L^p(D)},
\end{equation}
We refer to~\cite[Chapter 7]{Simader1972} for the proof of existence and uniqueness of weak solutions for the Laplace-Dirichlet problem in $W^{1,p}$ spaces.

For $k\geq 1$, the $k$-th order correction $\mean{u^k}$ satisfies the following problem:
\begin{eqnarray}
	\label{eq:uk}
	\nonumber
	&\mathbf{k}\textmd{\bf-th order correction BVP}&\\
	&\boxed{
	\begin{array}{l}
		\displaystyle
		\int_D\nabla\mean{u^k}\cdot\nabla v\ dx
		=-\sum_{j=1}^k\bc{k}{j}\int_D\mean{\nabla u^{k-j}Y^j}\cdot \nabla v\ dx 
		\quad\forall v\in W^{1,q}_0(D).
	\end{array}}&
\end{eqnarray} 
Equation \eqref{eq:uk} is obtained in two steps: (i) derive the problem satisfied by $u^k$, by taking derivatives with respect to $Y$ of the stochastic equation~\eqref{eq:darcy} (see \cite{Bonizzoni2014} and the references therein); (ii) apply the expected value to both sides of the obtained equation.

The function $\mean{\Grad u^{k-i}Y^i}$ appearing in the r.h.s. of~\eqref{eq:uk} is the diagonal of the $(i+1)$-points correlation function $\mean{\Grad u^{k-i}\otimes Y^{\otimes i}}$, where $\otimes$ denotes the tensor product.
In particular, it holds
$$
\mean{u^{k-i}Y^i}(x):=\left(\Trdiaggen{1}{i+1}\mean{u^{k-i}\otimes Y^{\otimes i}}\right)(x)
=\mean{u^{k-i}\otimes Y^{\otimes i}}
	(\underbrace{x,\ldots,x}_{(i+1)-\textmd{times}}),
$$
where 
\begin{itemize}
	\item $\mathrm{Tr}$ is the diagonal trace operator (it will be formally defined in Definition~\ref{def:trace});
	\item $\mean{u^{k-i}\otimes Y^{\otimes i}}(x,y_1,\ldots,y_i)$ is the $(i+1)$-point correlation function defined as
\begin{equation*}
	\mean{u^{k-i}\otimes Y^{\otimes i}}(x,y_1,\ldots,y_i):=
	\int_{\Omega}
		u^{k-i}(\omega,x)\otimes Y(\omega,y_1)\otimes\cdots 
		\otimes Y(\omega,y_i)\text{d}\mathbb P(\omega).
\end{equation*}
\end{itemize}
In the same way, 
$$
\mean{\Grad u^{k-i}Y^i}(x):=\left(\Trdiaggen{1}{i+1}\mean{\Grad u^{k-i}\otimes Y^{\otimes i}}\right)(x)
=\mean{\Grad u^{k-i}\otimes Y^{\otimes i}}
	(\underbrace{x,\ldots,x}_{(i+1)-\textmd{times}}),
$$
where $\mean{\Grad u^{k-i}\otimes Y^{\otimes i}}=\Grad\otimes \Id^{\otimes i}\mean{u^{k-i}\otimes Y^{\otimes i}}$, that is, the linear operator $\Grad\otimes \Id^{\otimes i}$ applies the gradient operator to the first variable $x$ and the identity operator to all other variables $y_j$ for $j=1,\ldots,i$. 

The correlation functions themselves satisfy the following recursion:

\newpage
\begin{center}
 \bf Recursion on the correlations
\end{center}
\begin{equation} 
 % \small
\framebox{\begin{minipage}{0.99\textwidth}
Given all lower order terms $\mean{u^{k-i-j}\otimes Y^{\otimes (i+j)}}$
for $j=1,\ldots,k-i$,
find $\mean{u^{k-i}\otimes Y^{\otimes i}}$ s.t.
\begin{eqnarray*}
&& \int_{D} 
    \left(\nabla\otimes\Id^{\otimes i}\right)
    \mean{u^{k-i}\otimes Y^{\otimes i}}(x,y_1,\ldots,y_i)
    \cdot
    \nabla v(x)\ \ddx \\
&& =-\sum_{j=1}^{k-i}\bc{k-i}{j}
    \hspace{-0.2cm}\ds{\int_{D} 
    \Trdiaggen{1}{j+1}\mean{\Grad u^{k-i-j} \otimes Y^{\otimes (i+j)}}(x,y_1,\ldots,y_i)
    \cdot
    \nabla v(x)\ \ddx,}
\end{eqnarray*}
$\forall v\in W^{1,q}_0(D)$, for all $y_1,\ldots,y_i\in D$.
\end{minipage}}
 \label{eq:recursion}
\end{equation}

Note that problem~\eqref{eq:uk} is a particular case of~\eqref{eq:recursion} for $i=0$, since $\mean{u^{k-0}\otimes Y^{\otimes 0}}=\mean{u^k}$.
Moreover, observe that $\mean{u^0\otimes Y^{\otimes k}}=u^0\otimes \mean{Y^{\otimes k}}$, since $u^0$ is deterministic, and it is fully characterized by the mean solution $u^0$ and the covariance structure of $Y$, which is an input of the problem.

The computation of the $k$-th order correction of the expected value of $u$ relies on the \emph{recursive} solution of BVPs of the type~\eqref{eq:recursion}, as summarized in Algorithm~\ref{al:Euk}.

\begin{algorithm}[h]
\caption{Computation of the $k$-th order correction $\mean{u^k}$}
\label{al:Euk}
    \begin{algorithmic}[1]
        \FOR{$k=0,\ldots,K$}
        \STATE Compute $u^0\otimes \mean{Y^{\otimes k}}$.
        \FOR{$i=k-1,k-2,\ldots,0$}
        \STATE Compute the $(i+1)$-point correlation function $\mean{u^{k-i}\otimes Y^{\otimes i}}$ (equation~\eqref{eq:recursion}).
        \ENDFOR 
        \STATE The result for $i=0$ is the $k$-th order correction $\mean{u^k}$ to the mean $\mean{u}$.
        \ENDFOR
    \end{algorithmic}
\end{algorithm}

\begin{table}[ht]
	\caption{$K$-th order approximation of the mean. The first column contains the input terms $\mean{u^0\otimes Y^{\otimes k}}$ and the first row contains the $k$-th order corrections $\mean{u^k}$, for $k=0,\ldots,K$. To compute $\mean{T^Ku(Y,x)}$, we need all the elements in the top left triangular part, that is, all elements in the $k$-th diagonal, for $k=0,\ldots,K$.}
	\begin{center}
	\begin{tabular}{c|c|c|c|c|c}
		$u^0$&$0$&$\mean{u^2}$&$0$&$\mean{u^4}$&$0$\\
		\hline
		$0$&$\mean{u^1\otimes Y}$&$0$&$\mean{u^3\otimes Y}$&$0$&$\iddots$\\
		\hline
		$u^0\otimes \mean{Y^{\otimes 2}}$&$0$&$\mean{u^2\otimes Y^{\otimes 2}}$&$0$&$\iddots$&$0$\\
		\hline
		$0$&$\mean{u^1\otimes Y^{\otimes 3}}$&$0$&$\iddots$&$0$&$\iddots$\\
		\hline
		$u^0\otimes \mean{Y^{\otimes 4}}$&$0$&$\iddots$&$0$&$\iddots$&$0$\\  
	\end{tabular}
	\end{center}
	\label{ricorsione_corr:table}
\end{table} 
Table~\ref{ricorsione_corr:table} illustrates the computational flow of the presented algorithm. Each non-zero correlation $\mean{u^{k-i}\otimes Y^{\otimes i}}$, with $i<k$, can be obtained only when all the preceding terms in the $k$-th diagonal have been already computed.
As a consequence, to derive the $K$-th order approximation $\mean{T^Ku}$, it is necessary to compute all the elements in the top left triangular part of the table.
Notice that, since we assumed $\mean{Y}(x)=0$, all the $(2k+1)$-point correlations of $Y$ vanish, and all odd diagonals are zero.

%%%%%%%%%%%%%%%%%%%%%%%%%%%%%%%%%%%%%%%%%%%%%%%%%%%%%%%%%%%%%%%%%%%%%%
\section{Banach space-valued mixed H\"older maps, and trace results}
\label{sec:Preliminaries} 
%%%%%%%%%%%%%%%%%%%%%%%%%%%%%%%%%%%%%%%%%%%%%%%%%%%%%%%%%%%%%%%%%%%%%%

Within this section, we introduce the notion of $V$-valued H\"older mixed regular map, $V$ being a Banach space, and we study the regularity of the diagonal trace of $V$-valued H\"older mixed regular maps.

%%%%%%%%%%%%%%%%%%%%%%%%%%%%%%%%%%%%%%%%%%%%%%%%%%%%%%%%%%%%%%%%%%%%%%
\subsection{Banach space-valued mixed H\"older spaces}
%%%%%%%%%%%%%%%%%%%%%%%%%%%%%%%%%%%%%%%%%%%%%%%%%%%%%%%%%%%%%%%%%%%%%%

\begin{definition}[Banach space-valued H\"older space]
\label{def:Banach-valued Holder space}
Let $V$ be a Banach space, $0<\gamma\leq 1$ be real, and $k\geq 1$ integer. The $V$-valued H\"older space with exponent $\gamma$, $\Cspace{0,\gamma}{\bar D^{\times k}}{V}$, consists of all continuous maps $\varphi=\varphi(x,y_1,\ldots,y_k):\bar D^{\times k}\rightarrow V$ with H\"older $\gamma$-regularity with respect to the variable $\vet y=(y_1,\ldots,y_k)$.
It is a Banach space with the norm
\begin{equation*}
	\norm{\varphi}{\Cspace{0,\gamma}{\bar D^{\times k}}{V}}
	:=\max\left\{
	\norm{\varphi}{\Cspace{0}{\bar D^{\times k}}{V}},
	\seminorm{\varphi}{\Cspace{0,\gamma}{\bar D^{\times k}}{V}}
	\right\}
\end{equation*}
with
\begin{equation*}
	\norm{\varphi}{\Cspace{0}{\bar D^{\times k}}{V}}
	:=\sup_{\vet y\in\bar D^{\times k}}
		\norm{\varphi(\cdot,\vet y)}{V}
\end{equation*}
and
\begin{equation*}
	\seminorm{\varphi}{\Cspace{0,\gamma}{\bar D^{\times k}}{V}}
		:= \sup_{\substack{\vet y\in\bar D^{\times k},
			\vet h\neq \vet{0} \\
			\textmd{s.t. }\vet{y+h} \in\bar D^{\times k}}}
	\frac{\norm{\varphi(\cdot,\vet y+\vet h)-\varphi(\cdot,\vet y)}{V}}
		{\norm{\vet h}{}^\gamma},
\end{equation*}
where
$$\vet h:=(h_1,\ldots,h_k)=(h_{1,1},\ldots,h_{1,d};h_{2,1},\ldots,h_{2,d};\ldots;h_{k,1},\ldots,h_{k,d})\in\R^{kd},$$ 
that is $h_j$ is a vector of $d$ components for each $j=1,\ldots,k$,
and $\norm{\cdot}{}$ denotes the Euclidean norm. 
In the following, the space $\Cspace{0,\gamma}{\bar D^{\times k}}{V}$ will be also denoted as $\Cspacepedex{0,\gamma}{y_1,\ldots,y_k}{\bar D^{\times k}}{V_x}$ or $\Cspacepedex{0,\gamma}{\vet y}{\bar D^{\times k}}{V_x}$.
\end{definition}

\begin{definition}
\label{def:holder_increment}
Let $h_j\neq 0$. The one-dimensional difference quotient $D^\gamma_{j,h_j}$ along the direction $j$ and with exponent $0<\gamma\leq 1$ of the function $v:\bar D^{\times k}\rightarrow \R$ is defined as
\begin{equation}
	\label{eq:one_dim_D}
	D^\gamma_{j,h_j} v(y_1,\ldots,y_k)
	:=\frac{v(y_1,\ldots,y_{j}+h_{j},\ldots,y_k)-v(y_1,\ldots,y_k)}
	{\norm{h_{j}}{}^\gamma}.
\end{equation}
\end{definition}

\begin{definition}
\label{def:mixed_holder_increment}
Given $\vet h=(h_1,\ldots,h_k)\in\R^{kd}$, we introduce $\vet i=\vet i(\vet h)$ as the vector containing the (non repeated) indices corresponding to the non-zero entries $h_j$ of $\vet h$, and $\vet i(\vet h)^c=\{1,\ldots,k\}\setminus\vet i(\vet h)$ (i.e., $h_j\neq (0,\ldots,0)$ for all $j\in\vet i(\vet  h)$, and $h_j=(0,\ldots,0)$ for all $j\in\vet i(\vet  h)^c$). 
The mixed difference quotient $D^{\gamma,mix}_{\vet i,\vet h}$ is defined as
\begin{equation}
	\label{eqLmixed_D}
	D^{\gamma,mix}_{\vet i,\vet h} := 
	\prod_{j=1}^{\norm{\vet h}{0}} D^\gamma_{i_j,h_{i_j}}
	% D^\gamma_{i_{\norm{\vet h}{0}}}\cdots D^\gamma_{i_1},
\end{equation}
where $\norm{\vet h}{0}:=\# \vet i(\vet h)$.
\end{definition}

In the following, when no confusion arises, we will denote the one-dimensional difference quotient also as $D^\gamma_{j}$, and the mixed different quotient as $D^{\gamma,mix}_{\vet i}$, omitting to specify the increment $\vet h$.

\begin{definition}[Banach space-valued mixed H\"older space]
\label{def:Banach-valued mixed Holder space}	
Let $V$ be a Banach space, $0<\gamma\leq 1$ be real, and $k\geq 1$ integer. The $V$-valued mixed H\"older space with exponent $\gamma$, $\Cspacemix{0,\gamma}{\bar D^{\times k}}{V}$, consists of all continuous maps $\varphi=\varphi(x,y_1,\ldots,y_k):\bar D^{\times k}\rightarrow V$ with H\"older $\gamma$-regularity in each variable $y_j$, $j=1,\ldots,k$, separately.
It is a Banach space with the norm
\begin{equation}
\label{eq:norm_mixed_holder}
	\norm{\varphi}{\Cspace{0,\gamma,mix}{\bar D^{\times k}}{V}}
	:=\max\left\{
	\norm{\varphi}{\Cspace{0}{\bar D^{\times k}}{V}},
	\seminorm{\varphi}{\Cspace{0,\gamma,mix}{\bar D^{\times k}}{V}}
	\right\}
\end{equation}
where $\norm{\cdot}{\Cspace{0}{\bar D^{\times k}}{V}}$ is as in Definition~\ref{def:Banach-valued Holder space}, and 
\begin{equation}
\label{eq:seminorm_mixed_holder}
\seminorm{\varphi}{\Cspace{0,\gamma,mix}{\bar D^{\times k}}{V}}
:=\max_{j=1,\ldots,k}
	\sup_{\substack{\vet y\in \bar D^{\times k},\ \vet{h}\neq \vet{0},\\
	\textmd{s.t. } \norm{\vet h}{0}=j\\
	\textmd{and }\vet{y+h}\in \bar D^{\times k}}}
	\norm{ D^{\gamma,mix}_{\vet i} \varphi(\cdot,\vet y)}{V},
\end{equation}
$D^{\gamma,mix}_{\vet i}$ being introduced in Definition~\ref{def:mixed_holder_increment}.
In the following, the space $\Cspacemix{0,\gamma}{\bar D^{\times k}}{V}$ will be also denoted as $\Cspacemixpedex{0,\gamma}{y_1,\ldots,y_k}{\bar D^{\times k}}{V_x}$ or $\Cspacemixpedex{0,\gamma}{\vet y}{\bar D^{\times k}}{V_x}$. 
\end{definition}

%%%%%%%%%%%%%%%%%%%%%%%%%%%%%%%%%%%%%%%%%%%%%%%%%%%%%%%%%%%%%%
\subsubsection{Banach space-valued H\"older spaces with higher regularity}
%%%%%%%%%%%%%%%%%%%%%%%%%%%%%%%%%%%%%%%%%%%%%%%%%%%%%%%%%%%%%%

Let $V$, $k$ and $\gamma$ as in Definition~\ref{def:Banach-valued Holder space}, and let $n\geq 1$ integer. Moreover, given a vector, denote as $\abs{\cdot}$ its $\ell_1$-norm. We define
\begin{equation}
    \label{eq:Cn_gamma_V}
    \Cspacepedex{n,\gamma}{\vet y}{\bar D^{\times k}}{V_x} = \left\{
    \begin{array}{c}
        \varphi\in\Cspacepedex{n}{\vet y}{\bar D^{\times k}}{V_x} \text{ s.t. } 		\forall\,\bm{\alpha}=(\alpha_1,\ldots,\alpha_k)\in\N^{kd}\\ 
		\text{ with } \abs{\bm\alpha}=\abs{\alpha_1}+\cdots+\abs{\alpha_k}\leq n\\
        \der{\bm\alpha} \varphi = \der{\alpha_1}_{y_1} \cdots \der{\alpha_k}_{y_k} \varphi
        \in\Cspacepedex{0,\gamma}{\vet y}{\bar D^{\times k}}{V_x}
    \end{array}\right\}.
\end{equation}

The space $\Cspacepedex{n,\gamma}{\vet y}{\bar D^{\times k}}{V_x}$ is a Banach space with seminorm 
\begin{equation}
    \label{eq:holder_n_V_seminorm}
    \seminorm{\varphi}{\Cspacepedex{n,\gamma}{\vet y}{\bar D^{\times k}}{V_x}}
    := \max_{\abs{\bm\alpha}= n}
    \norm{\der{\bm\alpha} \varphi(\cdot,\vet y) }{\Cspacepedex{0,\gamma}{\vet y}{\bar D^{\times k}}{V_x}}
\end{equation}
and norm
\begin{equation}
    \label{eq:holder_n_V_norm}
    \norm{\varphi}{\Cspacepedex{n,\gamma}{\vet y}{\bar D^{\times k}}{V_x}}
    := \max\left\{
	\norm{\varphi}{\Cspacepedex{n}{\vet y}{\bar D^{\times k}}{V_x}},
	\seminorm{\varphi}{\Cspacepedex{n,\gamma}{\vet y}{\bar D^{\times k}}{V_x}}
	\right\}.
\end{equation}

Moreover, we introduce the space
\begin{equation}
    \label{eq:mixed_Cn_V}
    \Cspacemixpedex{n}{\vet y}{\bar D^{\times k}}{V_x} = \left\{
    \begin{array}{c}
        \varphi:D^{\times k}\rightarrow V \text{ s.t. } 		\forall\,\bm{\alpha}=(\alpha_1,\ldots,\alpha_k)\in\N^{kd} \text{ with }\\
		\alpha_\ell=(\alpha_{\ell,1},\ldots,\alpha_{\ell,d})\in\N^d\ \text{and }
        0\leq \abs{\alpha_{\ell}}\leq n\ \forall\, \ell,\\
        \der{\bm\alpha} \varphi = \der{\alpha_1}_{y_1} \cdots \der{\alpha_k}_{y_k} \varphi
        \in\Cspacepedex{0}{\vet y}{\bar D^{\times k}}{V_x}
    \end{array}\right\},
\end{equation}
which is a Banach space with the norm
\begin{equation}
    \label{eq:mixed_Cn_V_norm}
    \norm{\varphi}{\Cspacemixpedex{n}{\vet y}{\bar D^{\times k}}{V_x}}
	:=\max_{\substack{(\bm\alpha_1,\ldots,\bm\alpha_k)\in\N^{kd}\\
        0\leq\abs{\bm\alpha_\ell}\leq n}} 
		\norm{\der{\bm\alpha}\varphi}{\Cspacepedex{0}{\vet y}{\bar D^{\times k}}{V_x}}.
\end{equation}

Finally, generalizing Definition~\ref{def:Banach-valued mixed Holder space}, we introduce the space $\Cspacemixpedex{n,\gamma}{\vet y}{\bar D^{\times k}}{V_x}$ as follows:
\begin{equation}
    \label{eq:mixed_holder_n_V}
    \Cspacemixpedex{n,\gamma}{\vet y}{\bar D^{\times k}}{V_x} = \left\{
    \varphi\in\Cspacemixpedex{n}{\vet y}{\bar D^{\times k}}{V_x} \text{ s.t. }
    \seminorm{\varphi}{\Cspacemixpedex{n,\gamma}{\vet y}{\bar D^{\times k}}{V_x}}<+\infty
    \right\},
\end{equation}
where 
\begin{equation}
    \label{eq:mixed_holder_n_V_seminorm}
    \seminorm{\varphi}{\Cspacemixpedex{n,\gamma}{\vet y}{\bar D^{\times k}}{V_x}}
    := \max_{j=1,\ldots, k}
    \sup_{\substack{\vet y\in \bar D^{\times k},\ \vet{h}\neq 0,\\
	\textmd{s.t. } \norm{\vet h}{0}=j\\
	\textmd{and }\vet{y+h}\in \bar D^{\times k}}} 	\max_{\substack{\bm\alpha=(\alpha_1,\ldots,\alpha_k)\in\N^{kd}\\
	0\leq\abs{\alpha_\ell}\leq n,\ \forall\, \ell=1,\ldots,k\\
    \abs{\bm{\alpha}_{\ell}}=n,\ \ell\in\vet i(\vet h)}}
    \norm{D^{\gamma,mix}_{\vet i,\vet h}
    \der{\bm\alpha} \varphi(\cdot,\vet y) }{V_x}. 
\end{equation}
It is a Banach space with the norm
\begin{equation}
    \label{eq:mixed_holder_n_V_norm}
    \norm{\varphi}{\Cspacemixpedex{n,\gamma}{\vet y}{\bar D^{\times k}}{V_x}}
    :=\max\left\{
	\norm{\varphi}{\Cspacemixpedex{n}{\vet y}{\bar D^{\times k}}{V_x}},
    \seminorm{\varphi}{\Cspacemixpedex{n,\gamma}{\vet y}{\bar D^{\times k}}{V_x}}\right\}.	
\end{equation}

%%%%%%%%%%%%%%%%%%%%%%%%%%%%%%%%%%%%%%%%%%%%%%%%%%%%%%%%%%%%%%%%%%%%%%
\subsubsection{Properties of Banach space-valued mixed H\"older spaces}
%%%%%%%%%%%%%%%%%%%%%%%%%%%%%%%%%%%%%%%%%%%%%%%%%%%%%%%%%%%%%%%%%%%%%%

The following proposition states the relation between H\"older spaces and mixed H\"older spaces.

\begin{proposition}
\label{prop:holder_inclusion}
Let $V$ be a Banach space, and $0<\gamma\leq 1$. Then,
\begin{equation}
	\label{eq:holder_inclusion}
	\Cspacepedex{0,\gamma}{\vet y}{\bar D^{\times k}}{V_x}
	\subset 
	\Cspacemixpedex{0,\gamma/k}{\vet y}{\bar D^{\times k}}{V_x}	
\end{equation}
for all $k\geq 2$.
\end{proposition}

\begin{proof}
We first prove \eqref{eq:holder_inclusion} for $k=2$. 
Let $\varphi\in\Cspacepedex{0,\gamma}{y_1,y_2}{\overline{D\times D}}{V_x}$. Then,
\begin{align}
	\nonumber
	\seminorm{\varphi}{\Cspacemixpedex{0,\gamma/2}{y_1,y_2}{\overline{ D\times D} }{V_x}}
	& = \max\left\{
	\sup_{y_1,y_2,h_1} \norm{D^{\gamma/2}_1\varphi(\cdot,y_1,y_2)}{V_x},
	\sup_{y_1,y_2,h_2} \norm{D^{\gamma/2}_2\varphi(\cdot,y_1,y_2)}{V_x},\right.\\
	\label{eq:pezzi}
	&\quad\left.
	\sup_{y_1,y_2,h_1,h_2} \norm{D^{\gamma/2}_1D^{\gamma/2}_2\varphi(\cdot,y_1,y_2)}{V_x}\right\}.	
\end{align}
We bound the three terms in~\eqref{eq:pezzi} separately. Observe that
\begin{align}
	\nonumber
	\sup_{y_1,y_2,h_1} \norm{D^{\gamma/2}_1\varphi(\cdot,y_1,y_2)}{V_x}
	& = \sup_{y_1,y_2,h_1} 
		\frac{\norm{\varphi(\cdot,y_1+h_1,y_2)-\varphi(\cdot,y_1,y_2)}{V_x}}
		{\norm{h_1}{}^{\gamma/2}}\\
	\nonumber
	& = \sup_{y_1,y_2,h_1} \norm{h_1}{}^{\gamma/2}
		\frac{\norm{\varphi(\cdot,y_1+h_1,y_2)-\varphi(\cdot,y_1,y_2)}{V_x}}
		{\norm{h_1}{}^{\gamma}}\\
	% & \leq \seminorm{\varphi}{\Cspacepedex{0,\gamma/2}{y_1,y_2}{\overline{D\times D}}{V_x}}\\
	\label{eq:parte1}
	& \leq \max\{1,diam(D)^{\gamma/2}\}
	\seminorm{\varphi}{\Cspacepedex{0,\gamma}{y_1,y_2}{\overline{D\times D}}{V_x}},
\end{align}
which is bounded by assumption, and the same holds for $\sup_{y_1,y_2,h_2} \norm{D^{\gamma/2}_2\varphi(\cdot,y_1,y_2)}{V_x}$. 
We focus now on the third term in~\eqref{eq:pezzi}.
Define
\begin{align*}
	& w(\cdot, y_1,y_2;h_1,h_2)
	:=D^{\gamma/2}_1D^{\gamma/2}_2\varphi(\cdot,y_1,y_2) 
	\norm{h_1}{}^{\gamma/2}\norm{h_2}{}^{\gamma/2}\\
	&\quad =\varphi(\cdot,y_1+h_1,y_2+h_2)-\varphi(\cdot,y_1+h_1,y_2)-\varphi(\cdot,y_1,y_2+h_2)-\varphi(\cdot,y_1,y_2).
\end{align*}
Hence, we have
\begin{align*}
	& \sup_{y_1,y_2,h_1,h_2} \norm{D^{\gamma/2}_1D^{\gamma/2}_2\varphi(\cdot,y_1,y_2)}{V_x}
	= \sup_{y_1,y_2,h_1,h_2}
		\frac{\norm{ w(\cdot,y_1,y_2;h_1,h_2) }{V_x}}
		{\norm{h_1}{}^{\gamma/2}\norm{h_2}{}^{\gamma/2}}\\
	& \quad \leq \max\left\{
		\sup_{\vet y,\norm{h_1}{}<\norm{h_2}{}} \frac{\norm{ w(\cdot,y_1,y_2;h_1,h_2) }{V_x}}
		{\norm{h_1}{}^{\gamma/2}\norm{h_2}{}^{\gamma/2}},
		\sup_{\vet y,\norm{h_1}{}\geq\norm{h_2}{}} \frac{\norm{ w(\cdot,y_1,y_2;h_1,h_2) }{V_x}}
		{\norm{h_1}{}^{\gamma/2}\norm{h_2}{}^{\gamma/2}}
	\right\}.
\end{align*}
We start considering 
\begin{align*}
	& \sup_{\vet y,\norm{h_1}{}<\norm{h_2}{}} 
		\frac{\norm{ w(\cdot,y_1,y_2;h_1,h_2) }{V_x}}
	{\norm{h_1}{}^{\gamma/2}\norm{h_2}{}^{\gamma/2}}\\
	& \quad \leq 
	\sup_{\vet y,\norm{h_1}{}<\norm{h_2}{}}
	\frac{1}{\norm{h_1}{}^{\gamma/2}\norm{h_2}{}^{\gamma/2}}
	\bigg(\norm{h_1}{}^\gamma
	\frac{\norm{\varphi(\cdot,y_1+h_1,y_2+h_2)-\varphi(\cdot,y_1,y_2+h_2)}{V_x}}{\norm{h_1}{}^\gamma}\\
	&\quad\quad + 
	\norm{h_1}{}^\gamma
	\frac{\norm{\varphi(\cdot,y_1+h_1,y_2)-\varphi(\cdot,y_1,y_2)}{V_x}}{\norm{h_1}{}^\gamma}\bigg)\\
	& \quad \leq
	\sup_{\vet{y},\norm{h_1}{}<\norm{h_2}{}}
	\frac{\norm{h_1}{}^{\gamma/2}}{\norm{h_2}{}^{\gamma/2}}
	\left(\norm{D^\gamma_{1}\varphi(\cdot,y_1,y_2+h_2)}{V_x}
	+\norm{D^{\gamma}_{1}\varphi(\cdot,y_1,y_2)}{V_x}\right)\\
	&\quad\leq 2 \norm{\varphi}{\Cspacepedex{0,\gamma}{y_1,y_2}{\overline{D\times D}}{V_x}}.
\end{align*}
The case $\norm{h_1}{}\geq\norm{h_2}{}$ is analogous.
Hence, we conclude \eqref{eq:holder_inclusion} for $k=2$.

In the general case, given $\vet h=(h_1,\ldots,h_k)$, let $\vet i(\vet h)$ as in Definition~\ref{def:mixed_holder_increment}, and 
$i^*\in\{\vet i(\vet h)\}$ such that $\norm{h_{i^*}}{}\leq \norm{h_{i_j}}{}$ for all $j$ such that $i^*\neq i_j$. Moreover, define $w(\cdot,\vet y;\vet h):=D^{\gamma,mix}_{\vet i} \varphi(\cdot,\vet y)\prod_{\ell=1}^j\norm{h_{i_\ell}}{}^{\gamma/j}$.
We bound each term of the seminorm~\eqref{eq:seminorm_mixed_holder} as follows:
\begin{align*}
	&\sup_{\vet y,\norm{h_{i^*}}{}\leq \norm{h_{i_j}}{}}
	\frac{\norm{w(\cdot,\vet y;\vet h)}{V_x} }
	{\prod_{\ell=1}^j\norm{h_{i_\ell}}{}^{\gamma/j}}
	\leq
	\sup_{\vet y,\norm{h_{i^*}}{}\leq \norm{h_{i_j}}{}}
	\frac{\norm{h_{i^*}}{}^\gamma }{\prod_{\ell=1}^j\norm{h_{i_\ell}}{}^{\gamma/j}}
	\frac{\norm{w(\cdot,\vet y;\vet h)}{V_x}}
	{\norm{h_{i^*}}{}^\gamma }\\
	& \quad \leq
	2^{j-1} \seminorm{\varphi}{\Cspacepedex{0,\gamma}{\vet y}{\bar D^{\times k}}{V_x}},
\end{align*}
and the inclusion~\eqref{eq:holder_inclusion} is then proved.
\end{proof}

\begin{proposition}
\label{prop:holder_mix}
The spaces $\Cspacepedex{0,\gamma}{y_2}{\bar D}{\Cspacepedex{0,\gamma}{y_1}{\bar D}{V_x}}$ and $\Cspacepedex{0,\gamma}{y_1}{\bar D}{\Cspacepedex{0,\gamma}{y_2}{\bar D}{V_x}}$ are isomorphic to the space $\Cspacemixpedex{0,\gamma}{y_1,y_2}{\overline{D\times D}}{V_x}$ for all $n\geq 0$ integer.
\end{proposition}

\begin{proof}
According to definition~\eqref{eq:mixed_holder_n_V_norm}, we have
\begin{align*}
    &\norm{\varphi}{\Cspacemixpedex{0,\gamma}{y_1,y_2}{\overline{D\times D}}{V_x}}
    =\max\left\{
	\norm{\varphi}{\Cspacemixpedex{0}{y_1,y_2}{\overline{D\times D}}{V_x}},
    \seminorm{\varphi}{\Cspacemixpedex{0,\gamma}{y_1,y_2}{\overline{D\times D}}{V_x}}
	\right\}\\
    & \quad = \max\Big\{
            \max_{y_1,y_2} \norm{\varphi(\cdot,y_1,y_2)}{V_x},
		\sup_{(y_1,y_2), h_1} \norm{D^\gamma_{1} \varphi(\cdot,y_1,y_2)}{V_x},\\
    & \hspace{2cm} 
        \sup_{(y_1,y_2), h_2} \norm{D^\gamma_{2} \varphi(\cdot,y_1,y_2)}{V_x},
        \sup_{(y_1,y_2),(h_1,h_2)} \norm{D^\gamma_{2}D^\gamma_{1}
			\varphi(\cdot,y_1,y_2)}{V_x}\Big\}.
\end{align*}  
On the other hand, we have
\begin{align*}
    &\norm{\varphi}{\Cspacepedex{0,\gamma}{y_2}{\bar D}{\Cspacepedex{0,\gamma}{y_1}{\bar D}{V_x}}}
    = \max\left\{
    \norm{\varphi}{\Cspacepedex{0}{y_2}{\bar D}{\Cspacepedex{0,\gamma}{y_1}{\bar D}{V_x}}},
    \seminorm{\varphi}{\Cspacepedex{0,\gamma}{y}{\bar D}{\Cspacepedex{0,\gamma}{y_1}{\bar D}{V_x}}}
    \right\},
\end{align*}
where
\begin{align*}
    &\norm{\varphi}{\Cspacepedex{0}{y_2}{\bar D}{\Cspacepedex{0,\gamma}{y_1}{\bar D}{V_x}}}
	= \max_{y_2} \norm{ \varphi(\cdot,\cdot,y_2)}{\Cspacepedex{0,\gamma}{y_1}{\bar D}{V_x}}\\
    & \quad
    = \max_{y_2} \max\left\{
		\norm{\varphi(\cdot,\cdot,y_2)}{\Cspacepedex{0}{y_1}{\bar D}{V_x}},
    	\seminorm{ \varphi(\cdot,\cdot,y_2)}{\Cspacepedex{0,\gamma}{y_1}{\bar D}{V_x}}\right\}\\
    & \quad
    = \max\left\{
        \max_{y_1,y_2}\norm{\varphi(\cdot,y_1,y_2)}{V_x},
        \max_{y_2}\sup_{y_1,h_1}
        \norm{D^\gamma_{1} \varphi(\cdot,y_1,y_2)}{V_x}
        \right\},
\end{align*}
and
\begin{align*}
    & \seminorm{\varphi}{\Cspacepedex{0}{y_2}{\bar D}{\Cspacepedex{0,\gamma}{y_1}{\bar D}{V_x}}}
    = \sup_{y_2,h_2} \norm{D^\gamma_{2} \varphi(\cdot,\cdot,y_2)}
			{\Cspacepedex{0,\gamma}{y_1}{\bar D}{V_x}}\\
    & \quad = \sup_{y_2,h_2}
    \max\left\{
    	\norm{D^\gamma_{2} \varphi(\cdot,\cdot,y_2)}
    		{\Cspacepedex{0}{y_1}{\bar D}{V_x}},
	    	\seminorm{D^\gamma_{2}\varphi(\cdot,\cdot,y_2)}
	    		{\Cspacepedex{0,\gamma}{y_1}{\bar D}{V_x}}\right\}\\
    & \quad =
    \max\left\{
    \max_{y_1} \sup_{y_2,h_2} 
        \norm{D^\gamma_{2} \varphi(\cdot,y_1,y_2)}{V_x},
    \sup_{y_2,h_2} \sup_{y_1,h_1}
        \norm{D^\gamma_{1} D^\gamma_{2} \varphi(\cdot,y_1,y_2)}{V_x}\right\}.
\end{align*}
Hence, we conclude that $\norm{\varphi}{\Cspacemixpedex{0,\gamma}{y_1,y_2}{\overline{D\times D}}{V_x}}=\norm{\varphi}{\Cspacepedex{0,\gamma}{y_2}{\bar D}{\Cspacepedex{0,\gamma}{y_1}{\bar D}{V_x}}}$. In the same way, it is possible to show that $\norm{\varphi}{\Cspacemixpedex{0,\gamma}{y_1,y_2}{\overline{D\times D}}{V_x}}=\norm{\varphi}{\Cspacepedex{0,\gamma}{y_1}{\bar D}{\Cspacepedex{0,\gamma}{y_2}{\bar D}{V_x}}}$.
\end{proof}

\begin{remark}
With small modifications to the proof, it is possible to prove that Proposition~\ref{prop:holder_inclusion} holds for H\"older spaces with higher regularity, yielding 
\begin{equation}
	\label{eq:holder_inclusion_n}
	\Cspacepedex{n,\gamma}{\vet y}{\bar D^{\times k}}{V_x}
	\subset 
	\Cspacemixpedex{n,\gamma/k}{\vet y}{\bar D^{\times k}}{V_x}	
\end{equation}
for all $k\geq 2$.
Moreover, Proposition~\ref{prop:holder_mix} generalizes to higher regularity and higher dimension, yielding 
\begin{equation}
	\label{eq:holder_mix_n}
	\Cspacemixpedex{n,\gamma}{\vet y}{\bar D^{\times k}}{V_x}
	\sim 
	\Cspacemixpedex{n,\gamma}{\vet y^\star}{\bar D^{\times(k-1)}}
		{\Cspacepedex{n,\gamma}{y_i}{\bar D}{V_x}}
	\quad\forall i=1,\ldots,k+1,
\end{equation}
where $\vet y^\star=(y_1,\ldots,y_{i-1},y_{i+1},\ldots,y_k)$.
\end{remark}

\begin{proposition}
	\label{prop:u_tp}
	Denote with $\Holdermix{0,\gamma}{\bar D^{\times k}}$ the space $\Cspacemix{0,\gamma}{\bar D^{\times k}}{\R}$. Then, it holds
	\begin{equation}
		\label{eq:u_tp}
		\norm{u}{\Holdermix{0,\gamma}{\bar D^{\times k}}}
		= \prod_{\ell=1}^k\norm{u_\ell}{\Holder{0,\gamma}{\bar D}},
	\end{equation}
	for all $u(y_1,\ldots,y_k):=u_1(y_1)\otimes\cdots\otimes u_k(y_k)\in \Holdermix{0,\gamma}{\bar D^{\times k}}$.
\end{proposition}

\begin{proof}
Using~\eqref{eq:norm_mixed_holder}, we have:
\begin{equation*}
	\norm{u}{\mcC^{0,\gamma,mix}(\bar D^{\times k})}
	=\max\left\{\norm{u_1\otimes\cdots\otimes u_k}
		{\mcC^0(\bar D^{\times k})},
	\seminorm{u_1\otimes\cdots\otimes u_k}
		{\mcC^{0,\gamma,mix}(\bar D^{\times k})}\right\}.
\end{equation*}
Observe that:
\begin{align*}
	& \norm{u_1\otimes\cdots\otimes u_k}{\mcC^0(\bar D^{\times k})}
	= \max_{(y_1,\ldots,y_k)\in\bar D^{\times k}} 
		\abs{u_1(y_1)\otimes\cdots\otimes u_k(y_k)}\\
	& \quad = \max_{(y_1,\ldots,y_k)\in\bar D^{\times k}} 
		\prod_{\ell=1}^k \abs{u_{\ell}(y_\ell)}
	= \prod_{\ell=1}^k \max_{y_\ell\in\bar D} \abs{u_{\ell}(y_\ell)}
	= \prod_{\ell=1}^k \norm{u_\ell}{\mcC^0(\bar D)}.
\end{align*}
We focus now on the seminorm of $u$:
\begin{align}
	\nonumber
	&\seminorm{u_1\otimes\cdots\otimes u_k}
		{\mcC^{0,\gamma,mix}(\bar D^{\times k})}
	= \max_{j=1,\ldots,k} 
		\sup_{\substack{\vet y\in \bar D^{\times k},\ \vet{h}\neq 0,\\
	\norm{\vet h}{0}=j,\ \vet y+\vet h\in \bar D^{\times k}}}
	\abs{D^{\gamma,mix}_{\vet i} u(y_1,\ldots,y_k)}\\
	\nonumber
	& \quad = \max_{j=1,\ldots,k}
		\sup_{\substack{\vet y\in \bar D^{\times k},\ \vet{h}\neq 0,\\
		\norm{\vet h}{0}=j,\ \vet y+\vet h\in\bar D^{\times k}}}
		\prod_{\ell\in\{1,\ldots,k\}\setminus\{\vet i\}}
			\abs{u_\ell(y_\ell)}
		\prod_{\ell\in\{\vet i\}}
			\abs{D^\gamma_{\ell}u_\ell(y_\ell)}\\
	\label{eq:prop_part1}
	& \quad = \max_{j=1,\ldots,k}
		\prod_{\substack{\ell\in\{1,\ldots,k\}\setminus\{\vet i\}\\
			\norm{\vet i}{0}=j}}
		\sup_{y_\ell\in\bar D}\abs{u_\ell(y_\ell)}
		\prod_{\substack{\ell\in\{\vet i\}\\
			\norm{\vet i}{0}=j}}
		\sup_{\substack{y_\ell\in\bar D,h_\ell\neq 0\\y_\ell+h_\ell\in\bar D}}
		\abs{D^\gamma_{\ell} u_\ell(y_\ell)}.
\end{align}
Choosing $j=k$, we have
\begin{align*}
	\eqref{eq:prop_part1}
	& \geq \prod_{\ell=1}^k \sup_{\substack{y_\ell\in\bar D,h_\ell\neq 0\\y_\ell+h_\ell\in\bar D}}
		\abs{D^\gamma_{\ell} u_\ell(y_\ell)} 
	= \prod_{\ell=1}^k \seminorm{u_\ell}{\Holder{0,\gamma}{\bar D}}.
\end{align*}
On the other hand, given $j^\star$ the index which realizes the maximum, we have
\begin{align*}
	\eqref{eq:prop_part1}
	& =
		\prod_{\substack{\ell\in\{1,\ldots,k\}\setminus\{\vet i\}\\
			\norm{\vet i}{0}=j^\star}}
		\sup_{x_\ell\in\bar D}\abs{u_\ell(y_\ell)}
		\prod_{\substack{\ell\in\{\vet i\}\\
			\norm{\vet i}{0}=j^\star}}
		\sup_{\substack{y_\ell\in\bar D,h_\ell\neq 0\\y_\ell+h_\ell\in\bar D}}
			\abs{D^\gamma_{\ell} u_\ell(y_\ell)}\\
	& =  \prod_{\substack{\ell\in\{1,\ldots,k\}\setminus\{\vet i\}\\
			\norm{\vet i}{0}=j^\star}}
		\norm{u_\ell}{\mcC^0(\bar D)}
		\prod_{\substack{\ell\in\{\vet i\}\\
			\norm{\vet i}{0}=j^\star}}
		\seminorm{u_\ell}{\Holder{0,\gamma}{\bar D}}
	\leq \prod_{\ell=1}^k \norm{u_\ell}{\Holder{0,\gamma}{\bar D}}.
\end{align*}
Hence, we have proved:
\begin{align*}
	\norm{u}{\mcC^{0,\gamma,mix}(\bar D^{\times k})}
	& \geq \max\left\{
		\prod_{\ell=1}^k \norm{u_\ell}{\mcC^0(\bar D)},
		\prod_{\ell=1}^k \seminorm{u_\ell}{\Holder{0,\gamma}{\bar D}}
	\right\}
	= \prod_{\ell=1}^k \norm{u_\ell}{\Holder{0,\gamma}{\bar D}},
\end{align*}
and
\begin{align*}
	\norm{u}{\mcC^{0,\gamma,mix}(\bar D^{\times k})}
	& \leq \max\left\{
		\prod_{\ell=1}^k \norm{u_\ell}{\mcC^0(\bar D)},
		\prod_{\ell=1}^k \norm{u_\ell}{\Holder{0,\gamma}{\bar D}}
	\right\}
	= \prod_{\ell=1}^k \norm{u_\ell}{\Holder{0,\gamma}{\bar D}},
\end{align*}
and~\eqref{eq:u_tp} follows.
\end{proof}

\begin{remark}
With small modifications to the proof, it is possible to prove that Proposition~\ref{prop:u_tp} holds for H\"older spaces with higher regularity, yielding to:
	\begin{equation}
		\label{eq:u_tp_n}
		\norm{u}{\Holdermix{n,\gamma}{\bar D^{\times k}}}
		= \prod_{\ell=1}^k\norm{u_\ell}{\Holder{n,\gamma}{\bar D}},
	\end{equation}
	for all $u(y_1,\ldots,y_k):=u_1(y_1)\otimes\cdots\otimes u_k(y_k)\in \Holdermix{n,\gamma}{\bar D^{\times k}}$.
\end{remark}

%%%%%%%%%%%%%%%%%%%%%%%%%%%%%%%%%%%%%%%%%%%%%%%%%%%%%%%%%%%%%%%%%%%%%%
\subsection{Diagonal trace of Sobolev space-valued mixed H\"older maps}
%%%%%%%%%%%%%%%%%%%%%%%%%%%%%%%%%%%%%%%%%%%%%%%%%%%%%%%%%%%%%%%%%%%%%%

In this section we focus on maps in the space $\Cspacemix{n,\gamma}{\bar D^{\times k}}{V}$, where $V$ is the Sobolev space $W_x^{m,p}(D)$.

\begin{definition}[Diagonal trace]
\label{def:trace}
Let $p,q,N$ be positive integers satisfying $1\leq p\leq q\leq N$, and let $v$ be a function of $N$ variables.
Then the diagonal trace function $\Trdiaggen{p}{q} v$ is a function of $N-(q-p)$ variables, defined as
\begin{align*}
\left(\Trdiaggen{p}{q}\right) v(x_1,\ldots,x_{p},x_{q+1},\ldots,x_N)
:=v(x_1,\ldots,x_{p-1},\underbrace{x_{p},\ldots,x_{p}}_{(q-p+1)-\textmd{times}},x_{q+1},\ldots,x_N).
\end{align*}
\end{definition}

\begin{proposition}
\label{prop:traccia_Wp}
Let $\varphi=\varphi(x,y_1,\ldots,y_k)\in\Cspacemixpedex{n,\gamma}{y_1,\ldots,y_k}{\bar D^{\times k}}{W_x^{m,p}(D)}$, with $D\subset\R^d$, $k\geq 1$ integer, $n\geq m\geq 0$ integers, $\gamma\in (0,1]$ and $p>\frac{2d}{\gamma}$.
Then, for all $j=2,\ldots,k+1$, and for all $(y_j,\ldots,y_k)\in D^{\times (k-j+1)}$, $(\Trdiag{j}\varphi)(x;y_j,\ldots,y_k)\in W_x^{m,p}(D)$.
In particular, there exists $C_{tr}>0$ such that
\begin{align}
    \nonumber
    &\norm{(\Trdiag{j}\varphi)(x;y_j,\ldots,y_k)}{W_x^{m,p}(D)}\\
    \label{eq:traccia_Wp_2}
    &\quad \leq C_{tr}^{j-1} \norm{\varphi(x,y_1,\ldots,y_{j-1};y_j,\ldots,y_k)}
    {\Cspacemixpedex{n,\gamma}{y_1,\ldots,y_{j-1}}{\bar D^{\times (j-1)}}{W^{m,p}_x(D)}},
\end{align}
for all $(y_j,\ldots,y_k)\in D^{\times (k-j+1)}$.\\
Moreover, $\Trdiag{j}\varphi\in\Cspacemixpedex{n,\gamma}{y_j,\ldots,y_k}{\bar D^{\times (k-j+1)}}{W^{m,p}(D)}$, and
\begin{equation}
    \label{eq:traccia_Wp_Cnmix}
	\norm{\Trdiag{j}\varphi}
		{\Cspacemixpedex{n,\gamma}{y_j,\ldots,y_k}{\bar D^{\times (k-j+1)}}{W^{m,p}(D)}}
	\leq C_{tr}^{j-1} \norm{\varphi}
		{\Cspacemixpedex{n,\gamma}{y_1,\ldots,y_{k}}{\bar D^{\times k}}{W^{m,p}_x(D)}}
\end{equation}
for all $j=2,\ldots,k+1$.
\end{proposition}

\begin{proof}
We prove the results in three steps.

\textbf{Step 1: inequality~\eqref{eq:traccia_Wp_2} for $k=1$ and $j=2$}\\
Let $\xi=\xi(x,y)\in\Cspacepedex{n,\gamma}{y}{\bar D}{W^{m,p}_x(D)}$, i.e., $x\mapsto\xi(x,\cdot)\in\Holder{n,\gamma}{\bar D}$ a.e., and $y\mapsto\xi(\cdot,y)\in W^{m,p}(D)$. Denote with $g(x):=\left(\Trdiag{2}\xi\right)(x)$ for all $x=(x_1,\ldots,x_d)\in D$. We want to show that $g\in W^{m,p}(D)$, i.e., $\hder{ \alpha}{x}g=\frac{\partial^{\abs{\alpha}} g}{\partial^{\alpha_1}x_1\cdots \partial^{\alpha_d}x_d}\in L^p(D)$ for all $\alpha=(\alpha_1,\ldots,\alpha_d)\geq 0$ with $\abs{\alpha}=\alpha_1+\ldots+\alpha_d\leq m$.

Let $\alpha$ be such that $\abs{\alpha}\leq m$, and let $x^{(i)}=(x^{(i)}_1,\ldots,x^{(i)}_d)\in D$ for $i=1,2$. Then, it holds
\begin{align*}
    &\norm{\hder{\alpha}{x} g(x)}{L^p_x(D)}
    = \norm{\frac{\partial^{\abs{\alpha}} g}
        {\partial^{\alpha_1}x_1\cdots \partial^{\alpha_d}x_d}}{L^p_x(D)}\\
    & \quad 
    = \norm{\sum_{t_1=0}^{\alpha_1} \cdots \sum_{t_d=0}^{\alpha_d}
        \bc{\alpha_1}{t_1}\cdots \bc{\alpha_d}{t_d}
        \frac{\partial^{\abs{\alpha}} \xi(x^{(1)}, x^{(2)})}
        {\partial^{t_1}x_1^{(1)}\partial^{\alpha_1-t_1}x_1^{(2)}\cdots
            \partial^{t_d}x_d^{(1)}\partial^{\alpha_d-t_d}x_d^{(2)}}
        \bigg|_{(x,x)}}
        {L^p_x(D)}\\
    & \quad
    \leq \sum_{t_1=0}^{\alpha_1} \cdots \sum_{t_d=0}^{\alpha_d}
        \bc{\alpha_1}{t_1}\cdots \bc{\alpha_d}{t_d}
        \norm{\frac{\partial^{\abs{\alpha}} \xi(x^{(1)}, x^{(2)})}
        {\partial^{t_1}x_1^{(1)}\partial^{\alpha_1-t_1}x_1^{(2)}\cdots
            \partial^{t_d}x_d^{(1)}\partial^{\alpha_d-t_d}x_d^{(2)}}
        \bigg|_{(x,x)}}
        {L^p_x(D)}.
\end{align*}
Denote $\hder{\alpha}{t} \xi(x^{(1)},x^{(2)}):=\frac{\partial^{\abs{\alpha}} \xi(x^{(1)}, x^{(2)})}{\partial^{t_1}x_1^{(1)}\partial^{\alpha_1-t_1}x_1^{(2)}\cdot\partial^{t_d}x_d^{(1)}\partial^{\alpha_d-t_d}x_d^{(2)}}$, where $t=(t_1,\ldots,t_d)$. 
Using the triangular inequality, we have
\begin{align}
    \nonumber
    & \norm{\hder{\alpha}{t}      
        \xi(x^{(1)},x^{(2)})\big|_{(x,x)}}{L^p_x(D)}\\
    \label{eq:Part1}
    & \quad
        \leq \sup_{y\in D} 
        \norm{\hder{\alpha}{t}      
        \xi(x^{(1)},x^{(2)})\big|_{(x,x)}
        -\hder{\alpha}{t}      
        \xi(x^{(1)},x^{(2)})\big|_{(x,y)}
        }{L^p_x(D)}\\
    \label{eq:Part2}
    & \quad\quad  + \sup_{y\in D} 
        \norm{\hder{\alpha}{t}     
        \xi(x^{(1)},x^{(2)})\big|_{(x,y)}
        }{L^p_x(D)}.
\end{align}
We bound first the term~\eqref{eq:Part1}. According to the Sobolev embedding theorem, if $sp>d$, then $W^{s,p}(D)\hookrightarrow \Holder{0,\beta}{D}$ for all $0<\beta<s-\frac{d}{p}$. Hence, there exists a positive constant $C_s$ such that
\begin{align}
	\label{eq:embedding}
	&\abs{\hder{\alpha}{t} \xi(x^{(1)},x^{(2)})\big|_{(x,y_1)}
        -\hder{\alpha}{t} \xi(x^{(1)},x^{(2)})\big|_{(x,y_2)}}\\
	\nonumber
	& \hspace{2cm}
	\leq C_s \abs{y_1-y_2}^{\beta}
	\left(\int_D\int_D \frac{\abs{ \hder{\alpha}{t} \xi(x^{(1)},x^{(2)})\big|_{(x,z_1)}
	-\hder{\alpha}{t} \xi(x^{(1)},x^{(2)})\big|_{(x,z_2)}}^p}
		{\abs{z_1-z_2}^{d+sp}} dz_1 dz_2\right)^{1/p}.
\end{align}
Using~\eqref{eq:embedding}, we have 
\begin{align*}
    & \sup_{y\in D} 
        \norm{\hder{\alpha}{t}\xi(x^{(1)},x^{(2)})\big|_{(x,x)}
        -\hder{\alpha}{t}\xi(x^{(1)},x^{(2)})\big|_{(x,y)}
        }{L^p_x(D)}^p\\
	& \quad 
	= \sup_{y\in D} \int_D \abs{\hder{\alpha}{ t}\xi(x^{(1)},x^{(2)})\big|_{(x,x)}
        -\hder{\alpha}{t}\xi(x^{(1)},x^{(2)})\big|_{(x,y)}}^pdx\\
	& \quad
	\leq C_s^p \int_D \sup_{y\in D} \abs{x-y}^{\beta p} 
		\left(\int_D\int_D 
		\frac{\abs{\hder{\alpha}{t}\xi(x^{(1)},x^{(2)})\big|_{(x,z_1)}
		-\hder{\alpha}{t}\xi(x^{(1)},x^{(2)})\big|_{(x,z_2)}}^p}
		{\abs{z_1-z_2}^{d+sp}} dz_1 dz_2\right) dx\\
	& \quad
	\leq C_s^p \abs{D}^{\beta p}
		\int_D\int_D\int_D \left(
		\frac{\abs{\hder{\alpha}{t}\xi(x^{(1)},x^{(2)})\big|_{(x,z_1)}
		-\hder{\alpha}{t}\xi(x^{(1)},x^{(2)})\big|_{(x,z_2)}}}
		{\abs{z_1-z_2}^{d/p+s}}\right)^p dz_1 dz_2 dx\\
	& \quad 
	= C_s^p \abs{D}^{\beta p}
		\int_D\int_D \frac{1}{\abs{z_1-z_2}^{d-\varepsilon}}
		\int_D \left(
		\frac{\abs{\hder{\alpha}{t}\xi(x^{(1)},x^{(2)})\big|_{(x,z_1)}
		-\hder{\alpha}{t}\xi(x^{(1)},x^{(2)})\big|_{(x,z_2)}}}
		{\abs{z_1-z_2}^{\varepsilon/p+s}}\right)^p dx dz_1 dz_2 \\
	& \quad
	\leq C_s^p \abs{D}^{\beta p}
		\seminorm{\hder{\alpha}{t}\xi}{\Cspace{0,\varepsilon/p+s}{\bar D}{L^p(D)}}^p
		\int_D\int_D \frac{1}{\abs{z_1-z_2}^{d-\varepsilon}} dz_1 dz_2\\
	& \quad
	\leq C_1(\varepsilon) C_s^p \abs{D}^{\beta p}
			\seminorm{\hder{\alpha}{t}\xi}{\Cspace{0,\varepsilon/p+s}{\bar D}{L^p(D)}}^p
\end{align*}
for all $0<\varepsilon<d$, with $C_1(\varepsilon):=\int_D\int_D \frac{1}{\abs{z_1-z_2}^{d-\varepsilon}} dz_1 dz_2<+\infty$. 
Hence, we have shown that 
\begin{align}
	\nonumber
	& \sup_{y\in D} 
	\norm{\hder{\alpha}{ t}\xi(x^{(1)},x^{(2)})\big|_{(x,x)}
	-\hder{\alpha}{ t}\xi(x^{(1)},x^{(2)})\big|_{(x,y)}}{L^p_x(D)}\\
	\label{eq:bound_part1}
	& \quad 
	\leq (C_1(\varepsilon))^{1/p} C_s \abs{D}^{s-d/p} 
	\seminorm{\hder{\alpha}{t}\xi}{\Cspace{0,\tilde\gamma}{\bar D}{L^p(D)}},
\end{align}
for any $s>\frac{d}{p}$, with $\tilde \gamma=\varepsilon/p+s$. Since $p>\frac{2d}{\gamma}$ and $\varepsilon<d$, by taking $s=\frac{\gamma}{2}>\frac{d}{p}$, we have $\tilde\gamma<\gamma$.

We bound now the second term~\eqref{eq:Part2}.
Since $\abs{\alpha}=\alpha_1+\cdots+\alpha_d\leq m\leq n$, then
\begin{equation}
    \eqref{eq:Part2}\leq \norm{\xi}{\Cspacepedex{n}{y}{\bar D}{W^{m,p}_x(D)}}
	\leq \norm{\xi}{\Cspacepedex{n,\gamma}{y}{\bar D}{W^{m,p}_x(D)}}.    
	\label{eq:bound_part2}
\end{equation}
Putting together~\eqref{eq:bound_part1} and~\eqref{eq:bound_part2}, we conclude~\eqref{eq:traccia_Wp_2} (for $k=1$ and $j=2$) with constant
$C_{tr}=2^m(C_1(\varepsilon)^{1/p}C_s\abs{D}^{s-d/p}+1)$.

\textbf{Step 2: inequality~\eqref{eq:traccia_Wp_2} for $k>1$ and $j=2,\ldots,k+1$}\\
Let $\varphi\in\Cspacemixpedex{n,\gamma}{y_1,\ldots,y_k}{\bar D^{\times k}}{W^{m,p}_x(D)}$, with $k> 1$ and $n\geq m$. 
We prove the proposition by induction on $j$. In Step 1, we have shown that the result holds for $j=2$, namely, for all $(y_2,\ldots,y_k)\in D^{\times(k-1)}$, $\Trdiag{2}\varphi(x;y_2,\ldots,y_k)\in W^{m,p}_x(D)$. In particular,   
\begin{align*}
    \norm{\Trdiag{2}\varphi(x;y_2,\ldots,y_k)}{W^{m,p}_x(D)}
    \leq C_{tr} \norm{\varphi(x,y_1;y_2,\ldots,y_k)}
        {\Cspacepedex{n,\gamma}{y_1}{\bar D}{W^{m,p}_x(D)}},
\end{align*}
for all $(y_2,\ldots,y_k)\in D^{\times(k-1)}$.

By induction, we assume that 
\begin{align*}
    % \label{eq:j_1_high_reh}
    \Trdiag{\ell}\varphi(x;y_\ell,\ldots,y_k)& \in W^{m,p}_x(D)\\
	\nonumber
    \norm{\Trdiag{\ell}\varphi(x;y_\ell,\ldots,y_k)}{W^{m,p}_x(D)}
    &\leq C_{tr} \norm{\Trdiag{\ell-1}\varphi(x;y_{\ell-1},\ldots,y_k)}{\Cspacepedex{n,\gamma}{y_{\ell-1}}{\bar D}{W^{m,p}_x(D)}}
\end{align*}
for all $\ell=3,\ldots,j$, and for all $(y_\ell,\ldots,y_k)\in D^{\times(k-\ell+1)}$.
Then, it holds
\begin{equation}
    \norm{\Trdiag{j}\varphi(x;y_j,\ldots,y_k)}{W^{m,p}_x(D)}
    \label{eq:j_2_high_reg}
    \leq C_{tr}^{j-1} \norm{\varphi(x,y_1,\ldots,y_{j-1};y_j,\ldots,y_k)}
        {\Cspacemixpedex{n,\gamma}{y_1,\ldots,y_{j-1}}{\bar D^{\times j}}{W^{m,p}_x(D)}},
\end{equation}
where we have used the isomorphism~\eqref{eq:holder_mix_n}.

Denote with $\mathbf y=(y_{j+1},\ldots,y_k)$. 
We bound $\norm{\Trdiag{j+1}\varphi(x;\mathbf y)}{W^{m,p}_x(D)}$ as follows:
\begin{align}
    \nonumber
    &\norm{\Trdiag{j+1}\varphi(x;\mathbf y)}{W^{m,p}_x(D)}
	=\norm{\Trdiag{j}\varphi(x,x;\mathbf y)}{W^{m,p}_x(D)}\\
    \label{eq:inequality_n}
    &\quad\leq\sup_{y_j\in D} 
        \norm{\Trdiag{j}\varphi(x,x;\mathbf y)-\Trdiag{j}\varphi(x,y_j;\mathbf y)}
            {W^{m,p}_x(D)}
        +\sup_{y_j\in D}\norm{\Trdiag{j}\varphi(x,y_j;\mathbf y)}{W^{m,p}_x(D)}.
\end{align}
Using inequality~\eqref{eq:j_2_high_reg} we bound the second term in the right hand side of~\eqref{eq:inequality_n} as:
\begin{align*}
    \sup_{y_j\in D}\norm{\Trdiag{j}\varphi(x,y_j;\mathbf y)}{W^{m,p}_x(D)}
    &\leq C_{tr}^{j-1} \sup_{y_j\in D} \norm{\varphi(x,y_1,\ldots,y_{j};\mathbf y)}{\Cspacemixpedex{n,\gamma}{y_1,\ldots,y_{j-1}}{\bar D^{\times (j-1)}}{W^{m,p}_x(D)}}\\
    &\leq C_{tr}^{j-1} \norm{\varphi(x,y_1,\ldots,y_{j};\mathbf y)}{\Cspacemixpedex{n,\gamma}{y_1,\ldots,y_{j}}{\bar D^{\times j}}{W^{m,p}_x(D)}}.    
\end{align*}
We bound the first term in the right hand side of~\eqref{eq:inequality_n} by proceeding as in the case $k=1$ and $j=2$:
\begin{align*}
    &\sup_{y_j\in D} 
        \norm{\Trdiag{j}\varphi(x,x;\mathbf y)-\Trdiag{j}\varphi(x,y_j;\mathbf y)}
        {W^{m,p}_x(D)}\\
    &\quad\leq C_{tr} \norm{\Trdiag{j}\varphi(x,y_j;\mathbf y)}{\Cspacepedex{n,\gamma}{y_j}{D}
        {W^{m,p}_x(D)}}\\
    &\quad \stackrel{\eqref{eq:j_2_high_reg}}{\leq} C_{tr}^j 
        \norm{\varphi(x,y_1,\ldots,y_j;\mathbf y)}{\Cspacemixpedex{n,\gamma}{y_1,\ldots,y_j}{\bar D^{\times j}}{W^{m,p}_x(D)}},
\end{align*}
and the conclusion holds.

\textbf{Step 3: mixed H\"older regularity of $\Trdiag{j}\varphi$}\\
Let $\xi(x,y_1,y_2)\in\Cspacemixpedex{n,\gamma}{y_1,y_2}{\bar D^{\times 2}}{W^{m,p}_x(D)}$. By applying the same steps as in Step 2 to the increment in the variable $y_2$ of the trace of $\xi$, $D^{\gamma}_{2}\left(\Trdiag{2}\xi\right)(x;y_2)$, we conclude~\eqref{eq:traccia_Wp_Cnmix} in the case $k=2$ and $j=2$. Then, by induction, we conclude~\eqref{eq:traccia_Wp_Cnmix} for any $k$ and $j$.
\end{proof}

%%%%%%%%%%%%%%%%%%%%%%%%%%%%%%%%%%%%%%%%%%%%%%%%%%%%%%%%%%%%%%%%%%%%%%
\section{Recursion on the correlations - analytical results}
\label{sec:Recursive equation for Eu} 
%%%%%%%%%%%%%%%%%%%%%%%%%%%%%%%%%%%%%%%%%%%%%%%%%%%%%%%%%%%%%%%%%%%%%%

This section is organized as follows. We first study the mixed H\"older regularity of the input of the recursion~\eqref{eq:recursion} , i.e., the $(k+1)$-points correlation function $\mean{u^0\otimes Y^{\otimes k}}$ (see Corollary~\ref{cor:mixed_reg_u0yk}). Then, in Section~\ref{sec:well-posedness_recursion}, we prove the well-posedness and regularity of the recursion itself.

%%%%%%%%%%%%%%%%%%%%%%%%%%%%%%%%%%%%%%%%%%%%%%%%%%%%%%%%%%%%%%%%%%%%%%
\subsection{Mixed H\"older regularity of the input of the recursion}
\label{sec:mixed_regularity_yk}
%%%%%%%%%%%%%%%%%%%%%%%%%%%%%%%%%%%%%%%%%%%%%%%%%%%%%%%%%%%%%%%%%%%%%%

The following proposition states the mixed H\"older regularity of the $(k+1)$-points correlation function $\mean{v\otimes Y^{\otimes k}}$, where $v$ belongs to a Banach space $V$.

\begin{proposition}
\label{prop:reg_CorrY}
Let $V$ be a Banach space, and $Y$ be a centered Gaussian random field such that $Y\in\Lspace{p}{\Holder{n,\gamma}{\bar D}}$, $n\geq 0$, for all $1<p<+\infty$. 
Then, for every $v\in V$ and every positive integer $k$, the $(k+1)$-points correlation $\mean{v\otimes Y^{\otimes k}}$ belongs to the H\"older space with mixed regularity $\Cspacemixpedex{n,\gamma}{\vet y}{\bar D^{\times k}}{V_x}$.
Moreover, it holds:
\begin{equation}
	\label{eq:norm_CorrY}
	\norm{\mean{v\otimes Y^{\otimes k}}}{\Cspacemixpedex{n,\gamma}{\vet y}{\bar D^{\times k}}{V_x}}
	= \norm{v}{V}
		\norm{\mean{Y^{\otimes k}}}{C^{n,\gamma,mix}(\bar D^{\times k})}.
\end{equation}
\end{proposition}

\begin{proof}
We prove that $\mean{v\otimes Y^{\otimes k}}\in \Cspacemixpedex{n,\gamma}{\vet y}{\bar D^{\times k}}{V_x}$ in two steps.

\textbf{Step 1: $\mean{Y^{\otimes k}}\in \Holdermix{n,\gamma}{\bar{D}^{\times k}}$}\\
We have to show that
\begin{enumerate}[label=(\roman*)]
    \item $\mean{Y^{\otimes k}}\in\Holdermix{n}{\bar{D}^{\times k}}$, i.e., 
    $\hder{\bm\alpha}{}\mean{Y^{\otimes k}}
    =\hder{\alpha_1}{x_1}\cdots\hder{\alpha_k}{x_k}\mean{Y^{\otimes k}}
    \in\Holder{0}{\bar D^{\times k}}$ for all $\bm\alpha=(\alpha_1,\ldots,\alpha_k)\in\N^{kd}$ with $0\leq\abs{\alpha_j}\leq n$, for $j=1,\ldots,k$.
    \label{enum:i}  
    \item $\hder{\bm\alpha}{} \mean{Y^{\otimes k}} \in \Holdermix{0,\gamma}{\bar{D}^{\times k}}$, for all $\bm{\alpha}=(\alpha_1,\ldots,\alpha_k)\in\N^{kd}$ with $\abs{\alpha_j}=n$, for some $j=1,\ldots,k$. 
    \label{enum:ii}
\end{enumerate}

Let us start with~\ref{enum:i}. Fix $\bm\alpha=(\alpha_1,\ldots,\alpha_k)\in\N^{kd}$ with $0\leq \abs{{\alpha}_j}\leq n$, for $j=1,\ldots,k$. Then,
\begin{align}
\nonumber
\norm{\hder{\bm\alpha}{} \mean{Y^{\otimes k}}}{\Holder{0}{\bar{D}^{\times k}}}
& = \max_{\vet y\in\bar{D}^{\times k}} \abs{ \hder{\bm\alpha}{}\mean{Y^{\otimes k} }(\vet y)}\\
\nonumber
& = \max_{\vet y\in\bar{D}^{\times k}} \abs{\hder{\alpha_1}{y_1}\cdots \hder{\alpha_k}{y_k}
    \mean{Y^{\otimes k} }(y_1,\ldots,y_k)}\\
\nonumber
& = \max_{\vet y\in\bar{D}^{\times k}} \abs{\mean{
    \hder{\alpha_1}{y_1} Y(y_1)\otimes\cdots\otimes\hder{\alpha_k}{y_k} Y(y_k)}}\\
\label{eq:reg1}
& \leq \max_{\vet y\in\bar{D}^{\times k}} 
	\mean{ \abs{ \hder{\alpha_1}{y_1}Y(y_1)
		 \otimes \cdots \otimes
		 \hder{\alpha_k}{y_k}Y(y_k) }}.
\end{align}
Using the H\"older inequality, we get
\begin{align*}
\eqref{eq:reg1} 
& \leq \max_{\vet y\in\bar{D}^{\times k}} 
    \prod_{i=1}^k \left( \mean{ \abs{\hder{\alpha_i}{y_i}Y(y_i)}^k }\right)^{1/k}
\leq \prod_{i=1}^k \max_{y_i\in\bar D}
	\left( \mean{ \abs{ \hder{\alpha_i}{y_i}Y(y_i)}^k }\right)^{1/k}.
\end{align*}
Observe that 
\begin{align*}
& \max_{y_i\in\bar D}
    \left( \mean{ \abs{ \hder{\alpha_i}{y_i} Y(y_i)}^k }\right)^{1/k}
= \left( \max_{y_i\in\bar D} \mean{ \abs{ \hder{\alpha_i}{y_i}Y(y_i)}^k } \right)^{1/k}\\
& \quad \leq
\left(\mean{\max_{y_i\in\bar D} \abs{ \hder{\alpha_i}{y_i} Y(y_i)}^k}\right)^{1/k}
= \left(\mean{\left(\max_{y_i\in\bar D}
    \abs{ \hder{\alpha_i}{y_i}Y(y_i)}\right)^k }\right)^{1/k}\\
&\quad\leq\left(\mean{\norm{Y}{\mcC^n(\bar D)}^k}\right)^{1/k}
= \norm{Y}{\Lspace{k}{\Holder{n}{\bar D}}}.
\end{align*}
We conclude that
\begin{equation*}
    \prod_{i=1}^k \max_{y_i\in\bar D}
	\left( \mean{ \abs{ \hder{\alpha_i}{y_i} Y(y_i)}^k }\right)^{1/k}
    \leq \norm{Y}{\Lspace{k}{\Holder{n}{\bar D}}}^k < +\infty.
\end{equation*}

We prove now~\ref{enum:ii}. Let $\bm\alpha=(\alpha_1,\ldots,\alpha_k)$ with $\abs{\alpha_j}=n$ for some $j=1,\ldots,k$. Using Definitions~\ref{def:holder_increment} and~\ref{def:Banach-valued mixed Holder space}, we have
\begin{align}
	\nonumber
	&\seminorm{ \hder{\bm\alpha}{} \mean{Y^{\otimes k}}} {\Holdermix{0,\gamma}{\bar D^{\times k}}}
	= \max_{j=1,\ldots,k} \sup_{\substack{\vet y,\vet h\\\norm{\vet h}{0}=j}}
	\abs{ D^{\gamma,mix}_{\vet i} \hder{\bm\alpha}{} \mean{Y^{\otimes k}}}\\
	\nonumber
	& \quad =
	\max_{j=1,\ldots,k} \sup_{\substack{\vet y,\vet h\\\norm{\vet h}{0}=j}}
	\abs{ D^{\gamma}_{i_j}\cdots D^{\gamma}_{i_1} \hder{\bm\alpha}{} \mean{Y^{\otimes k}} }\\
	\nonumber
	& \quad =
	\max_{j=1,\ldots,k} \sup_{\substack{\vet y,\vet h\\\norm{\vet h}{0}=j}}
	\abs{ D^{\gamma}_{i_j}\cdots D^{\gamma}_{i_1} 
	\mean{ \hder{\alpha_1}{y_1}Y(y_1)\otimes\cdots\otimes \hder{\alpha_k}{y_k}Y(y_k) } }\\
	\label{eq:reg2}
	& \quad =
	\max_{j=1,\ldots,k} \sup_{\substack{\vet y,\vet h\\\norm{\vet h}{0}=j}}
    \abs{ \mean{
	\bigotimes_{\ell\in\vet i(\vet h)}
    \frac{\hder{\alpha_{\ell}}{y_{\ell}} Y(y_\ell+h_\ell)-\hder{\alpha_\ell}{y_\ell} Y(y_\ell)}{\norm{h_\ell}{}^{\gamma}}
	\cdot \bigotimes_{\ell'\in\vet i(\vet h)^c} \hder{\alpha_{\ell'}}{y_{\ell'}} Y(y_{\ell'})
}}.
\end{align}

Proceeding as in the proof of~\ref{enum:i}, we conclude
\begin{align*}
\eqref{eq:reg2} 
& \leq 
\max_{j=1,\ldots,k}
	\prod_{\ell\in\vet i(\vet h)}
    \left(
    \mean{ \sup_{\substack{\vet y,\vet h\\\norm{\vet h}{0}=j}}
    \abs{ \frac{\hder{\alpha_\ell}{y_\ell} Y(y_\ell+h_\ell)-\hder{\alpha_\ell}{y_\ell}Y(y_\ell)}
    {\norm{h_\ell}{}^{\gamma}}}^k}\right)^{1/k}\\
	&\quad
	\prod_{\ell'\in\vet i(\vet h)^c}
    \left(
    \mean{ \abs{\hder{\alpha_{\ell'}}{y_{\ell'}} Y(y_{\ell'})}^{k}}\right)^{1/(k)}
	\\
& \leq \norm{Y}{\Lspace{k}{\Holder{n,\gamma}{\bar D}} }^k<+\infty. 
\end{align*}

\textbf{Step 2: $\mean{v\otimes Y^{\otimes k}}\in \Cspacemixpedex{n,\gamma}{\vet y}{\bar D^{\times k}}{V_x}$}\\
It is enough to observe that
\begin{align*}
	\abs{\mean{v\otimes Y^{\otimes k}}}_{\Cspacemixpedex{n,\gamma}{\vet y}{\bar D^{\times k}}{V_x}}
	&\quad = 
	\abs{v\otimes \mean{Y^{\otimes k}}}_{\Cspacemixpedex{n,\gamma}{\vet y}{\bar D^{\times k}}{V_x}}\\
	&\quad = \norm{v}{V} \abs{\mean{Y^{\otimes k}}}_{\Holdermix{n,\gamma}{\bar D^{\times k}}}
	<+\infty.
\end{align*}

It remains us to show equality~\eqref{eq:norm_CorrY}. By definition, it holds:
\begin{align*}
	& \norm{\mean{v\otimes Y^{\otimes k}}}{\Cspacepedex{n,mix}{\vet y}{\bar D^{\times k}}{V_x}}
	= \max_{\bm\alpha} \max_{\vet y} 
		\norm{\der{\bm\alpha}\mean{v\otimes Y^{\otimes k}}(\cdot,\vet y)}{V_x}\\
	& \quad = \max_{\bm\alpha} \max_{\vet y} 
		\norm{v(\cdot)\otimes \der{\bm\alpha}\mean{Y^{\otimes k}}(\vet y)}{V_x}
	= \max_{\bm\alpha} \max_{\vet y} \norm{v}{V_x} \abs{\der{\bm\alpha}\mean{Y^{\otimes k}}}\\
	& \quad = \norm{v}{V_x} \norm{\der{\bm\alpha}\mean{Y^{\otimes k}}}{C^{n,mix}_{\vet y}(\bar D^{\times k})}.
\end{align*}
In the same way, it is possible to show that
\begin{equation*}
	\seminorm{\mean{v\otimes Y^{\otimes k}}}{\Cspacemixpedex{n,\gamma}{\vet y}{\bar D^{\times k}}{V_x}}
	= \norm{v}{V_x} \seminorm{\der{\bm\alpha}\mean{Y^{\otimes k}}}{C^{n,\gamma,mix}_{\vet y}(\bar D^{\times k})},
\end{equation*}
and equality~\eqref{eq:norm_CorrY} follows.
\end{proof}

\begin{corollary}
\label{cor:mixed_reg_u0yk}
Applying Proposition~\ref{prop:reg_CorrY} with $v=u^0\in W^{1,p}(D)$, we have 	$\mean{u^0\otimes Y^{\otimes k}}\in\Cspacemix{n,\gamma}{\bar D^{\times k}}{W^{1,p}(D)}$.
\end{corollary}

%%%%%%%%%%%%%%%%%%%%%%%%%%%%%%%%%%%%%%%%%%%%%%%%%%%%%%%%%%%%%%%%%%%%%%
\subsection{Well-posedness and regularity of the recursion}
\label{sec:well-posedness_recursion}
%%%%%%%%%%%%%%%%%%%%%%%%%%%%%%%%%%%%%%%%%%%%%%%%%%%%%%%%%%%%%%%%%%%%%%

To lighten the notation, from now on we denote the $k$-th order correction $\mean{u^k}$ with $\corr{k}$, and the $(i+1)$-points correlation $\mean{u^{k-i}\otimes Y^{\otimes i}}(x,y_1,\ldots,y_i)$ with $\corr{k-i,i}$.

\begin{theorem}[Well-posedness of the recursion]
\label{th:wellposedness}
Let $D\subset \R^d$, such that $\partial D\in C^1$, and $Y\in\Lspace{s}{\Holder{0,\gamma}{\bar D}}$ for all $1\leq s<+\infty$. Let $f\in L^p(D)$ for $p>\frac{2d}{\gamma}$, and $1<q<\infty$ such that $\frac{1}{p}+\frac{1}{q}=1$. 
Then, for any $i=0,\ldots,k-1$, the Laplace-Dirichlet problem: Given $\corr{k-i-j,i+j}$ for $j=1,\ldots,k-i$, find $w(\cdot,\vet y)\in W^{1,p}_{0,x}(D)$ such that, for all $\vet y:=(y_1,\ldots,y_i)\in D^{\times i}$,
\begin{equation}
  \label{eq:prb_k}
    \int_{D} 
    \left(\nabla\otimes\Id^{\otimes i}\right)w(x,\vet y)
    \cdot
    \nabla v(x)\ \ddx
    = \mathcal L_{\vet y}(v)
    \quad\forall v\in W^{1,q}_0(D)
\end{equation}
has a unique solution for all $i=k,k-1,\ldots,0$, with 
\begin{equation}
\label{eq:wellposedness_bound}
    \norm{w(\cdot,\vet y)}{W^{1,p}_{0,x}(D)}
    \leq C \norm{\mathcal L_{\vet y}}{(W^{1,q}_0)^\star},
\end{equation}
where $C>0$ is independent of $\vet y$, and the linear form $\mathcal L_{\vet y}:W^{1,q}_{0}(D)\rightarrow\R$ is defined as
\begin{equation}
    \label{eq:L_form}
    \mathcal L_{\vet y}(v):=
    -\sum_{j=1}^{k-i}\bc{k-i}{j}
    \int_{D} 
    \Trdiaggen{1}{j+1}\Grad_x\corr{k-i-j,i+j}(x,\vet y)
    \cdot
    \nabla v(x)\ \ddx.
\end{equation}
Moreover, the unique solution belongs to the space $\Cspacemixpedex{0,\gamma}{y_1,\ldots,y_i}{\bar D^{\times i}}{W^{1,p}_{0,x}(D)}$ and coincides with $\corr{k-i,i}$.
\end{theorem}
  
\begin{proof}
We prove the theorem by induction. Let $k=2$ and $i=1$. The problem we handle with is: given $\corr{0,2}$, find $w(\cdot,y)\in W^{1,p}_{0,x}(D)$ s.t., for all $y\in D$,
\begin{equation}
\label{eq:prb_k2}
    \int_{D} \left(\nabla\otimes\Id^{\otimes i}\right)w(x,y)
    \cdot\nabla v(x)\ \ddx
    = \mathcal L_{y}(v)
    \quad\forall v\in W^{1,q}_0(D),
\end{equation}
where $\mathcal L_{y}(v):=-\int_{D} \Trdiaggen{1}{2}\Grad_x\corr{0,2}(x,y)\cdot\nabla v(x)\ \ddx$. 

\textbf{Step 1: well-posedness of problem~\eqref{eq:prb_k2}}\\
We have to show that $\mathcal L_y\in (W^{1,q}_{0})^\star$. 
Since $\partial D\in C^1$ and $f\in L^p(D)$, then $u^0\in W^{1,p}(D)$, as stated in Section~\ref{sec:analytical_derivation_Eu}. 
Applying Proposition~\ref{prop:reg_CorrY} with $n=0$, we have $\Grad_x\corr{0,2}\in \Cspacemixpedex{0,\gamma}{y_1,y_2}{\overline{D\times D}}{L^p(D)}$. 
Applying Proposition~\ref{prop:traccia_Wp} with $n=0$, we get $\Trdiag{2} \Grad_x\corr{0,2}\in \Cspacepedex{0,\gamma}{y_2}{D}{L^p_x(D)}$, and, in particular, 
$$
C_{\mathcal L}:=\sup_{y_2\in D}
\norm{\Trdiag{2} \Grad_x\corr{0,2}}{L^p_x(D)}<\infty.
$$ 
Hence, by the H\"older inequality, we have
\begin{equation*}
    \abs{\mathcal L_y(v)}
    \leq \norm{\Trdiag{2} \Grad_x\corr{0,2}}{L^p_x(D)}
    \norm{\Grad v}{L^q_x(D)}
    \leq C_{\mathcal L} \norm{v}{W^{1,q}(D)},
\end{equation*}
so that $\mathcal L_y\in (W^{1,q}_{0})^\star$ for all $y\in D$. Thanks to~\cite[Chapter 7]{Simader1972}, we conclude that problem~\eqref{eq:prb_k2} has a unique solution $w(\cdot,y)\in W^{1,p}_0(D)$ for every $y\in D$. Moreover, there exists a positive constant $C=C(p,d,D)$ such that
\begin{equation*}
  \norm{w(\cdot,y)}{W^{1,p}_0(D)}\leq C\norm{\mathcal L_y}{(W^{1,q}_0(D))^\star}
  \leq C\ C_{\mcL}.
\end{equation*}

\textbf{Step 2: H\"older regularity of $w(x,\cdot)$}\\
Let us consider the difference between problem~\eqref{eq:prb_k2} in $y+h$ and in $y$:
\begin{align}
  \nonumber
  &\int_D (\nabla\otimes \Id) (w(x,y+h)-w(x,y))\cdot\nabla v(x)\ dx\\
  \label{eq:prb_k2_yh}
  &\quad = \int_D (\Trdiaggen{1}{2} \Grad_x\corr{0,2}(x,y+h)
    -\Trdiaggen{1}{2} \Grad_x\corr{0,2}(x,y))\cdot\nabla v(x) \ dx,
\end{align}
for all $v\in W^{1,q}_0(D)$. Following the same procedure as in Step 1, we conclude that problem~\eqref{eq:prb_k2_yh} is well-posed, and
\begin{equation}
  \label{eq:prb_k2_yh_bound}
  \norm{w(\cdot,y+h)-w(\cdot,y)}{W^{1,p}_0(D)}
  \leq C\norm{\mathcal L_{y+h}-\mathcal L_y}{(W^{1,q}_0(D))^\star}.
\end{equation}
Hence, we have
\begin{align*}
 &\seminorm{w}{\Cspacepedex{0,\gamma}{y}{\bar D}{W^{1,p}_0(D)}}
  = \sup_{y,h} \frac{1}{\norm{h}{}^{\gamma}} \norm{w(\cdot,y+h)-w(\cdot,y)}{W^{1,p}_0(D)}\\
 & \quad \stackrel{\eqref{eq:prb_k2_yh_bound}}{\leq }
  C \sup_{y,h} \frac{1}{\norm{h}{}^{\gamma}}
    \norm{\mathcal L_{y+h}-\mathcal L_y}{(W^{1,q}_0(D))^\star}\\
 & \quad = 
 C \sup_{y,h} \frac{1}{\norm{h}{}^{\gamma}}
  \sup_{\substack{v\in W^{1,q}_0(D)\\\norm{v}{W^{1,q}_0(D)}=1}}
  \bigg| \int_D (\Trdiaggen{1}{2} \Grad_x\corr{0,2}(x,y+h)
	 -\Trdiaggen{1}{2} \Grad_x\corr{0,2})\cdot\nabla v(x) \ dx \bigg|\\
 & \quad \leq 
 C \sup_{y,h} \frac{1}{\norm{h}{}^{\gamma}}
  \norm{ \Trdiaggen{1}{2} \Grad_x\corr{0,2}(\cdot,y+h) 
  	- \Trdiaggen{1}{2} \Grad_x\corr{0,2}(\cdot,y) }{ L^p_x(D) } \\
 & \quad = C 
 \sup_{y,h} 
  \norm{ D^{\gamma}_{y,h} \Trdiaggen{1}{2} \Grad_x\corr{0,2}(\cdot,y) }{ L^p_x(D) }
	\leq C \seminorm{ \Trdiaggen{1}{2} \corr{0,2} }
    { \Cspacepedex{0,\gamma}{y}{\bar D}{W^{1,p}_0(D)} }\\
 & \quad \stackrel{\eqref{eq:traccia_Wp_Cnmix}}{\leq }
 	C C_{tr} \norm{\corr{0,2}}{ \Cspacemixpedex{0,\gamma}{}{\overline{D\times D}}{W^{1,p}_0(D)} } < +\infty,
\end{align*}
so that $w\in \Cspacepedex{0,\gamma}{y}{\bar D}{W^{1,p}_{0,x}(D)}$.
Moreover, since $\corr{1,1}$ solves problem~\eqref{eq:prb_k2} for every  $y\in D$, then $\corr{1,1}\in\Cspacepedex{0,\gamma}{y}{\bar D}{W^{1,p}_0(D)}$ is the unique solution of~\eqref{eq:prb_k2}.

We perform now the induction step. Let $k\geq 2$ and $0\leq i\leq k -1$ be fixed, and assume that 
$\corr{k-i-j,i+j}
\in\Cspacepedex{0,\gamma,mix}{y_1,\ldots,y_{i+j}}{\bar D^{\times (i+j)}}{W^{1,p}_{0}(D)}$,
for $j=1,\ldots,k-i$. 

\textbf{Step 1: well-posedness of problem~\eqref{eq:prb_k}}\\
We have to show that $\mathcal L_{\vet y}$ as in~\eqref{eq:L_form} is in $(W^{1,q}_{0})^\star$. 
Since $\corr{k-i-j,i+j}\in\Cspacemixpedex{0,\gamma}{y_1,\ldots,y_{i+j}}{\bar D^{\times (i+j)}}{W^{1,p}_{0}(D)}$, then 
$\Trdiaggen{1}{j+1}\Grad_x\corr{k-i-j,i+j}\in 
\Cspacemixpedex{0,\gamma}{y_1,\ldots,y_{i}}{\bar D^{\times i}}{L^p(D)}$, and, in particular,
$$
C_{\mathcal L,j}:=\sup_{y_1,\ldots,y_{i}\in D^{\times i}}
\norm{\Trdiaggen{1}{j+1}\Grad_x\corr{k-i-j,i+j}}{L^p_x(D)}<\infty.
$$
Hence, by the H\"older inequality, we have $\abs{\mathcal L_{\vet y}(v)}\leq C_{\mcL} \norm{v}{W^{1,q}(D)}$, with $C_\mcL:=\sum_{j=1}^{k-i}\bc{k-i}{j}C_{\mathcal L,j}$,
so that $\mathcal L_{\vet y}\in (W^{1,q}_{0})^\star$. Thanks to~\cite[Chapter 7]{Simader1972}, we conclude that problem~\eqref{eq:prb_k2} has a unique solution $w(\cdot,\vet y)\in W^{1,p}_0(D)$ for a.e. $\vet y\in D^{\times i}$. Moreover, it holds
\begin{equation*}
  \norm{w(\cdot,\vet y)}{W^{1,p}_0(D)}\leq C\norm{\mathcal L_{\vet y}}{(W^{1,q}_0(D))^\star}
  \leq C C_{\mcL}.
\end{equation*}

\textbf{Step 2: H\"older regularity of $w(x,\cdot)$}\\
By considering the problem solved by $D^{\gamma,mix}_{\vet i}w(x,\vet y)$, we have
\begin{align}
	\nonumber
  &  \norm{ D^{\gamma,mix}_{\vet i}w(\cdot,\vet y) }{W^{1,p}_0(D)}\\
  \nonumber
  & \quad  \leq C \sup_{\substack{v\in W^{1,q}_0(D)\\\norm{v}{W^{1,q}_0(D)}=1}}
	\abs{\sum_{j=1^{k-i}}\bc{k-i}{j} 
		 \int_D D^{\gamma,mix}_{\vet i} \Trdiaggen{1}{j+1} \Grad \corr{k-i-j,i+j}\cdot\Grad v dx }\\
	\label{eq:prb_k_yh_bound}
  & \quad\leq
  C \sum_{j=1}^{k-i}\bc{k-i}{j} 
	\norm{D^{\gamma,mix}_{\vet i} \Trdiaggen{1}{j+1} \Grad \corr{k-i-j,i+j}(\cdot,\vet y)}{L^p}.
\end{align}
Hence, we have
\begin{align*}
 & \seminorm{w}{\Cspacemixpedex{0,\gamma}{\vet y}{\bar D^{\times i}}{W^{1,p}_0(D)}}
  = \max_{\ell=1,\ldots,i} \sup_{\vet y,\vet h,\norm{\vet h}{0}=\ell}
    \norm{ D^{\gamma,mix}_{\vet i} w(\cdot,\vet y) } { W^{1,p}_0(D) } \\
 &\quad\stackrel{\eqref{eq:prb_k_yh_bound}}{\leq}
	C \max_{\ell=1,\ldots,i} \sup_{\vet y,\vet h,\norm{\vet h}{0}=\ell}
	\sum_{j=1}^{k-i}\bc{k-i}{j} 
		\norm{D^{\gamma,mix}_{\vet i} \Trdiaggen{1}{j+1} 
		\Grad \corr{k-i-j,i+j}(\cdot,\vet y)}{L^p}\\
 &\quad\leq C \sum_{j=1}^{k-i}\bc{k-i}{j}
 		\seminorm{\Trdiaggen{1}{j+1} \Grad \corr{k-i-j,i+j}}{\Cspacepedex{0,\gamma,mix}{\vet y}{\bar D^{\times i}}{L^p(D)}}\\
  &\quad\leq C \sum_{j=1}^{k-i}\bc{k-i}{j} C_{tr}^j 
		\norm{ \corr{k-i-j,i+j}}{\Cspacepedex{0,\gamma,mix}{y1,\ldots,y_{i+j}}{\bar D^{\times(i+j)}}{W^{1,p}(D)}}<+\infty.
\end{align*}
In particular, since $\corr{k-i,i}$ solves problem~\eqref{eq:prb_k} for a.e. $\vet y\in D^{\times i}$, then $\corr{k-i,i}\in\Cspacemixpedex{0,\gamma}{\vet y}{\bar D^{\times i}}{W^{1,p}_0(D)}$ is the unique solution of~\eqref{eq:prb_k}.
\end{proof}

\begin{theorem}[Regularity of the recursion]
\label{thm:regularity}
Let $D\subset\R^d$ such that $\partial D\in C^{2+r}$, $r\geq 0$.
Let $f\in W^{r,p}(D)$, and $Y\in\Lspace{s}{\Holder{n,\gamma}{\bar D}}$, for all $1\leq s<\infty$ and $n\geq r+1$. 
Then the correlation $\corr{k-i,i}\in\Cspacemixpedex{n,\gamma}{y_1,\ldots,y_i}{\bar D^{\times i}}{ W^{2+r,p}(D)\cap W^{1,p}_0(D)}$ for all $i=k,k-1,\ldots,0$. Moreover, there exists a positive constant $C_{reg}$ independent of $\vet y=(y_1,\ldots,y_i)$, such that
\begin{equation}
	\label{eq:regularity}
	\norm{\corr{k-i,i}(\cdot,\vet y)}{ W^{2+r,p}(D)}
	\leq C_{reg} \norm{\mcL_{\vet y}}{(W^{r,q})^*},
\end{equation}
where $\mcL_{\vet y}$ has been introduced in~\eqref{eq:L_form}.
\end{theorem}

\begin{proof}
We prove the theorem by induction. Let $k=2$ and $i=1$.
Since $f\in W^{r,p}(D)$, we have $u^0\in W^{1,p}_0(D)\cap W^{2+r,p}(D)$ (see~\cite[Chapter 9]{Simader1972}).
Using the assumption $Y\in\Lspace{s}{\Holder{n,\gamma}{\bar D}}$ and Proposition~\ref{prop:reg_CorrY}, we have 
$$
\corr{0,2}\in \Cspacemixpedex{n,\gamma}{y_1,y_2}{\overline{D\times D}}{W^{1,p}_0(D)\cap W^{2+r,p}(D)},$$ 
so that $\Grad_x\corr{0,2}\in \Cspacemixpedex{n,\gamma}{y_1,y_2}{\overline{D\times D}}{W^{1+r,p}(D)}.
$ 
Applying Proposition~\ref{prop:traccia_Wp}, we have $\Trdiag{2} \Grad\corr{0,2}\in \Cspacepedex{n,\gamma}{y}{\bar D}{W^{1+r,p}(D)}$. 
Following the same reasoning as in the proof of Theorem~\ref{th:wellposedness}, we have that $\corr{1,1}\in \Cspacepedex{n,\gamma}{y}{\bar D}{W^{2+r,p}(D)}$. Finally, since we have already shown that $\corr{1,1}\in \Cspacepedex{n,\gamma}{y}{\bar D}{W^{1,p}_0(D)}$, we conclude the thesis.

We perform now the induction step. Assume that the correlation 
$$
\corr{k-i-j,i+j}
\in\Cspacepedex{n,\gamma,mix}{y_1,\ldots,y_{i+j}}{\bar D^{\times (i+j)}}{W^{1,p}_0(D)\cap W^{r+2,p}(D)},
$$
for $j=1,\ldots,k-i$. 
Applying Proposition~\ref{prop:traccia_Wp}, we have
$$
\Trdiag{j+1} \Grad\corr{k-i-j,i+j}
\in\Cspacepedex{n,\gamma,mix}{y_1,\ldots,y_{i}}{\bar D^{\times i}}{W^{r+1,p}(D)}.
$$
Following the same reasoning as in the proof of Theorem~\ref{th:wellposedness}, we conclude that
$\corr{k-i,i}\in \Cspacemixpedex{n,\gamma}{y_1\ldots,y_i}{\bar D^{\times i}}{W^{2+r,p}(D)}$. 
Finally, since we have already shown that 
\\$\corr{k-i,i}\in \Cspacemixpedex{n,\gamma}{y_1\ldots,y_i}{\bar D^{\times i}}{W^{1,p}_0(D)}$, we conclude the thesis.

Finally, the upper bound~\eqref{eq:regularity} follows from~\cite[Chapter 9]{Simader1972}, observing that $\mcL_{\vet y}\in (W^{1+r,p})^*$.
\end{proof}

\begin{proposition}
\label{prop:bound_recursion}
Under the assumptions of Theorem~\ref{thm:regularity}, it holds
\begin{equation}
	\label{eq:bound_recursion}
	\norm{\corr{k-i,i}}{\Cspacemixpedex{n,\gamma}{y_1,\ldots,y_i}{\bar D^{\times i}}{W^{2+r,p}(D)}}
	\leq \lambda_{k-i}
	\norm{\corr{0,k}}{\Cspacemixpedex{n,\gamma}{y_1,\ldots,y_k}{\bar D^{\times k}}{W^{2+r,p}(D)}}
\end{equation}
for all $i\leq k$, where the coefficients $\{\lambda_{k-i}\}_{i=1}^k$ are defined by recursion as $\lambda_0:=1$ and $\lambda_{k-i}:=C_{reg} \sum_{j=1}^{k-i}\bc{k-i}{j} C_{tr}^{j}\,\lambda_{k-i-j}$ for $i<k$, the constants $C_{reg},\,C_{tr}$ being introduced in Theorem~\ref{thm:regularity} and Proposition~\ref{prop:traccia_Wp}, respectively.
\end{proposition}

\begin{proof}
Let $k$ be fixed. We prove the Theorem by induction on $i$.
If $i=k$, bound~\eqref{eq:bound_recursion} holds as an equality.
Let now $i<k$ fixed. By induction, we assume
\begin{equation}
	\label{eq:ind_ass}
	\norm{\corr{k-\ell,\ell}}
		{\Cspacemixpedex{n,\gamma}{y_1,\ldots,y_\ell}{\bar D^{\times \ell}}{W^{2+r,p}_0(D)}}
	\leq \lambda_{k-\ell}
	\norm{\corr{0,k}}{\Cspacemixpedex{n,\gamma}{y_1,\ldots,y_k}{\bar D^{\times k}}{W^{2+r,p}_0(D)}},
\end{equation}
for all $i+1\leq \ell\leq k-1$. 
Thanks to~\eqref{eq:regularity} and Proposition~\ref{prop:traccia_Wp}, it holds:
\begin{align*}
	& \norm{\corr{k-i,i}}
		{\Cspacemixpedex{n,\gamma}{y_1,\ldots,y_i}{\bar D^{\times i}}{W^{2+r,p}_0(D)}}\\
	& \quad \leq C_{reg} \sum_{j=1}^{k-i}\bc{k-i}{j} C_{tr}^j
	\norm{\corr{k-i-j,i+j}}
	{\Cspacemixpedex{n,\gamma}{y_1,\ldots,y_{i+j}}{\bar D^{\times(i+j)}}{W^{2+r,p}_0(D)}}.
\end{align*}
Using the assumption~\eqref{eq:ind_ass}, we have
\begin{align*}
	& \norm{\corr{k-i,i}}
	{\Cspacemixpedex{n,\gamma}{y_1,\ldots,y_i}{\bar D^{\times i}}{W^{2+r,p}_0(D)}}\\
	& \quad 
	\leq C_{reg} \sum_{j=1}^{k-i}\bc{k-i}{j} C_{tr}^j\, \lambda_{k-i-j} 
		\norm{\corr{0,k}}
		{\Cspacemixpedex{n,\gamma}{y_1,\ldots,y_k}{\bar D^{\times k}}{W^{2+r,p}_0(D)}}\\
	& \quad 
	= \lambda_{k-i}\norm{\corr{0,k}}
		{\Cspacemixpedex{n,\gamma}{y_1,\ldots,y_k}{\bar D^{\times k}}{W^{2+r,p}_0(D)}}.
\end{align*}
\end{proof}

%%%%%%%%%%%%%%%%%%%%%%%%%%%%%%%%%%%%%%%%%%%%%%%%%%%%%%%%%%%%%%%%%%%%
\section{Recursion on the correlations - sparse discretization}
\label{sec:Sparse discretization}
%%%%%%%%%%%%%%%%%%%%%%%%%%%%%%%%%%%%%%%%%%%%%%%%%%%%%%%%%%%%%%%%%%%%

Within this section we aim at deriving a discretization for the recursive problem~\eqref{eq:recursion}. In particular, the differential operator in the spatial variable $x$ will be discretized by the finite element method, whereas the parametric dependence on the variable $\vet y$ will be approximated by a sparse interpolation technique.
To this end, we first recall some preliminary results on the standard finite element projector $\pi_h$, and the sparse interpolant operator $\spinterp{L}$.

%%%%%%%%%%%%%%%%%%%%%%%%%%%%%%%%%%%%%%%%%%%%%%%%%%%%%%%%%%%%%%%%%%%%
\subsection{Finite element projector}
%%%%%%%%%%%%%%%%%%%%%%%%%%%%%%%%%%%%%%%%%%%%%%%%%%%%%%%%%%%%%%%%%%%%

Given a regular triangulation $\T_h$ of the domain $D$ with discretization parameter $h>0$, we denote with $\mathbb P_\mu(\T_h)$ the standard conforming finite element space of degree $\mu\geq 1$ defined on $\T_h$. There holds:
\begin{equation}
	\label{eq:Pmu_approx}
	\min_{v\in\mathbb P_\mu(\T_h)}\norm{u-v}{W^{1,p}(D)}
		\leq C_{fem} h^\beta \seminorm{u}{W^{2+r,p}(D)}
	\quad \forall\,u\in W^{2+r,p}(D),
\end{equation}
with $\beta=\min\{\mu,2+r\}-1$.

Let $\pi_h:W^{1,p}_0(D)\cap W^{2+r,p}(D)\rightarrow \mathbb P_\mu(\T_h)$ be the finite element projector, i.e., the operator which associates $u$ to its finite dimensional approximation via the finite element method. There holds (see~\cite[Chapter 8]{Brenner2007}):
\begin{equation}
	\label{eq:estimate_Pmu}
	\norm{u-\pi_h u}{W^{1,p}(D)}
	\leq C_{\pi_h} h^\beta \seminorm{u}{W^{2+r,p}(D)}
	\quad\forall u\in W^{2+r,p}(D),
\end{equation}
where $C_{\pi_h}>0$ is independent of $h$.

%%%%%%%%%%%%%%%%%%%%%%%%%%%%%%%%%%%%%%%%%%%%%%%%%%%%%%%%%%%%%%%%%%%%
\subsection{Sparse interpolant operator}
%%%%%%%%%%%%%%%%%%%%%%%%%%%%%%%%%%%%%%%%%%%%%%%%%%%%%%%%%%%%%%%%%%%%

Let $\{V_\ell\}_{l\geq 0}$ be a dense sequence of nested finite dimensional subspaces of $\Holder{0,\gamma}{\bar D}$, and let the discretization parameter $h_\ell$ of $V_\ell$ be $h_\ell:=\frac{h_{\ell-1}}{2}$, so that $h_\ell=h_0\, 2^{-\ell}$. 
Denote with $\{a_j^\ell\}_{j=1}^{N_\ell}\subset D$ a set of interpolation points unisolvent in $V_\ell$, and with $\{\xi_j^\ell\}_{j=1}^{N_\ell}$ the Lagrangian basis of $V_\ell$ such that $\xi_j^\ell(a_i^\ell)=\delta_{i,j}$ for all $i,j=1,\ldots,N_\ell$. Moreover, let $P_\ell:\Holder{0,\gamma}{\bar D}\rightarrow V_\ell$ be the Lagrangian interpolation operator, that is $P_\ell(v)=\sum_{j=1}^{N_\ell} v(a_j^{\ell}) \xi_j^\ell$, and note that $P_\ell$ is a projector, too.
Assume that $P_\ell$ satisfies the following property:
\begin{equation}
	\label{eq:approx_Pl}
	\norm{u-P_\ell u}{\Holder{0,\gamma}{\bar D}}
	\leq C\, h_\ell^s\norm{u}{\Holder{n,\gamma}{\bar D}} 
	\quad \forall u\in \Holder{n,\gamma}{\bar D},
\end{equation}
where $C>0$ is independent of $h_\ell$, and $s>0$. 

Following~\cite{Nobile2008}, we define the sparse interpolation operator as follows. Let $\Delta_\ell:=P_\ell-P_{\ell-1}$ be the difference operator. 
Given $k,\, L$ positive integers, the \emph{sparse interpolation operator of level $L$} is defined as: 
\begin{equation}
\label{eq:sparse_projector}
\spinterp{L,k}:=\sum_{\substack{\bm \ell=(\ell_1,\ldots,\ell_k)\in\N^k\\
	\abs{\bm \ell}\leq L}}
	\bigotimes_{j=1}^k \Delta_{\ell_j}.
\end{equation}
The sparse interpolation operator $\spinterp{L,k}$ maps the H\"older space with mixed regularity $\Holdermix{0,\gamma}{\bar D^{\times k}}$ onto the sparse tensor product space $\spfe{L,k}$, defined as
\begin{equation}
	\label{eq:stp_subspace}
	\spfe{L,k}:=
	\bigcup_{\substack{\bm{\ell}=(\ell_1,\ldots,\ell_k)\in\N^k\\\abs{\bm \ell}\leq L}} 
	\bigotimes_{j=1}^k V_{\ell_j}.
\end{equation}

The application of the sparse interpolation operator $\spinterp{L,k}$ to a function implies the evaluation of the function itself in a finite set of points - the \emph{sparse grid} - denoted as $\mathcal H_{L,k}$.
To lighten the notation, the sparse interpolation operator $\spinterp{L,k}$ will be simply denoted as $\spinterp{L}$, when no confusion occurs.

\begin{proposition}
\label{prop:approx_properties}
Let $k$ be a positive integer and $W$ a Banach space.
Then it holds:
\begin{equation}
	\label{eq:approx_properties}
	\norm{\spinterp{L,k}u-u}
	{\Cspacemixpedex{0,\gamma}{\vet y}{\bar D^{\times k}}{W_x}}
	\leq C_{\spinterp{L,k}}\, h_L^{s(1-\tau)}
	\norm{u}{\Cspacemixpedex{n,\gamma}{\vet y}{\bar D^{\times k}}{W_x}}
\end{equation}
for all $u=u(x,\vet y)\in \Cspacemixpedex{n,\gamma}{\vet y}{\bar D^{\times k}}{W_x}$, with $\vet y=(y_1,\ldots,y_k)\in \bar D^{\times k}$, where $0<\tau<1$, and $C_{\spinterp{L,k}}$ is a positive constant independent of $h_L$ (but blowing up when $\tau\rightarrow 0$).
\end{proposition}

\begin{proof}
The bound~\eqref{eq:approx_properties} is derived by standard computations (see, e.g.,~\cite{Bungartz2004}). For completeness, we report here all the steps.

Denote with $\hat{\vet y}_i\in\bar D^{\times (k-1)}$ the vector $(y_1,\ldots,y_{i-1},y_{i+1},\ldots,y_k)$.
We start giving an upper bound for the norm $\norm{\Delta_\ell\otimes \Id^{\otimes (k-1)}u(\cdot,\hat{\vet y}_1)}{\Cspacemixpedex{0,\gamma}{y_1}{\bar D}{W_x}}$. Using the triangular inequality and~\eqref{eq:approx_Pl}, we have
\begin{align}
	\nonumber
	& \norm{\Delta_{\ell}\otimes \Id^{\otimes (k-1)}u(\cdot,\hat{\vet y}_1)}{\Cspacemixpedex{0,\gamma}{y_1}{\bar D}{W_x}}\\
	\nonumber
	&\quad\leq
	\norm{(\interp{\ell}-\Id)\otimes \Id^{\otimes (k-1)}u(\cdot,\hat{\vet y}_1)}{\Cspacemixpedex{0,\gamma}{y_1}{\bar D}{W_x}}\\
	\nonumber
	&\quad\quad
	+ \norm{(\interp{\ell-1}-\Id)\otimes \Id^{\otimes (k-1)}u(\cdot,\hat{\vet y}_1)}{\Cspacemixpedex{0,\gamma}{y_1}{\bar D}{W_x}}\\
	\nonumber
	&\quad\leq 
	\left(C h_\ell^s+C h_{\ell-1}^s)\right)
	\norm{u(\cdot,\hat{\vet y}_1)}{\Cspacemixpedex{n,\gamma}{y_1}{\bar D}{W_x}}\\
	\nonumber
	&\quad\leq 
	2 C h_{\ell-1}^s
	\norm{u(\cdot,\hat{\vet y}_1)}{\Cspacemixpedex{n,\gamma}{y_1}{\bar D}{W_x}}\\
	\label{eq:a}
	&\quad\leq 
	2 C h_{0}^s\, 2^{-s(\ell+1)}
	\norm{u(\cdot,\hat{\vet y}_1)}{\Cspacemixpedex{n,\gamma}{y_1}{\bar D}{W_x}},
\end{align}
where we have used that $h_{\ell}\leq h_{\ell-1}$, and $h_{\ell-1}=h_0\,2^{-\ell-1}$

Using~\eqref{eq:a}, it follows:
\begin{align*}
	\nonumber
	& \norm{\Delta_{\ell_1}\otimes \Delta_{\ell_2}\otimes\Id^{\otimes (k-2)}u}{\Cspacemixpedex{0,\gamma}{y_1,y_1}{\overline{D\times D}}{W_x}}\\
	& \quad
	= \norm{(\Delta_{\ell_1}\otimes\Id^{\otimes (k-1)})
	\otimes
	(\Id\otimes\Delta_{\ell_2}\otimes\Id^{\otimes (k-2)})
	u}{\Cspacemixpedex{0,\gamma}{y_1,y_1}{\overline{D\times D}}{W_x}}\\
	&\quad \leq 
	4 C^2 h_{0}^{2s}\, 2^{-s(\ell_1+1)(\ell_2+1)}
	\norm{u}{\Cspacemixpedex{n,\gamma}{y_1,y_2}{\overline{D\times D}}{W_x}}.
\end{align*}

By recursion, we have
\begin{align}
	\nonumber
	&\norm{\Delta_{\ell_1}\otimes\cdots\otimes \Delta_{\ell_k}u}
		{\Cspacemixpedex{0,\gamma}{\vet y}{\bar D}{W}}\\
	\label{eq:b}
	&\quad 
	\leq 2^k C^k \left(\frac{h_0}{2}\right)^{sk}
		 2^{-s\abs{\bm\ell}}
		\norm{u}{\Cspacemixpedex{n,\gamma}{\vet y}{\bar D}{W_x}}.
\end{align}
By~\eqref{eq:b}, it follows that the series $\sum_{\bm\ell\in\N^k}\otimes_{n=1}^k\Delta_{\ell_n}u$ is absolutely convergent, and that $\sum_{\abs{\bm\ell}\leq L}\otimes_{n=1}^k\Delta_{\ell_n}u$ converges to $u$ as $L\rightarrow\infty$ in $\Cspacemixpedex{0,\gamma}{\vet y}{\bar D^{\times k}}{W_x}$.

Finally, we have
\begin{align*}
	& \norm{\spinterp{L}u-u}
		{\Cspacemixpedex{0,\gamma}{\vet y}{\bar D^{\times k}}{W_x}}
	= \norm{\sum_{\abs{\bm \ell}\leq L}\bigotimes_{n=1}^k \Delta_{\ell_n}
		u-u}
		{\Cspacemixpedex{0,\gamma}{\vet y}{\bar D^{\times k}}{W_x}}\\
	& \quad = \norm{\sum_{\abs{\bm \ell}> L}
		\bigotimes_{n=1}^k \Delta_{\ell_n}
		u}{\Cspacemixpedex{0,\gamma}{\vet y}{\bar D^{\times k}}{W_x}}
	\leq \sum_{\abs{\bm \ell}> L} 
		\norm{\bigotimes_{n=1}^k \Delta_{\ell_n}
		u}{\Cspacemixpedex{0,\gamma}{\vet y}{\bar D^{\times k}}{W_x}}\\
	& \quad 
	\leq 2^{k} C^k \left(\frac{h_0}{2}\right)^{sk}
		\left(\sum_{\abs{\bm \ell}> L} 2^{-s\abs{\bm\ell}}\right)
		\norm{u}{\Cspacemixpedex{n,\gamma}{\vet y}{\bar D}{W_x}}.
\end{align*}
In~\cite[Lemma 6.10]{Bonizzoni2014a} the authors prove that
\begin{equation*}
	\sum_{\abs{\bm\ell}>L} 2^{-s\abs{\bm\ell}}
	\leq \left(\frac{1}{1-2^{-s\tau}}\right)^k
	2^{-Ls(1-\tau)}
	= \left(\frac{1}{1-2^{-s\tau}}\right)^k
	h_0^{s(\tau-1)}\,h_L^{s(1-\tau)},
\end{equation*}
with $0<\tau<1$. 
Hence, we conclude~\eqref{eq:approx_properties} with 
\begin{equation*}
	C_{\spinterp{L,k}}=2^{k} \left(\frac{h_0}{2}\right)^{sk} h_0^{s(\tau-1)} C^k \left(\frac{1}{1-2^{-s\tau}}\right)^k.
\end{equation*}
\end{proof}

\begin{remark}
The result proved in Theorem~\ref{prop:approx_properties} holds whenever $P_\ell$ is any operator fulfilling~\eqref{eq:approx_Pl}.
\end{remark}

%%%%%%%%%%%%%%%%%%%%%%%%%%%%%%%%%%%%%%%%%%%%%%%%%%%%%%%%%%%%%%%%%%%%
\subsection{Sparse discretization of the recursion}
%%%%%%%%%%%%%%%%%%%%%%%%%%%%%%%%%%%%%%%%%%%%%%%%%%%%%%%%%%%%%%%%%%%%

As highlighted in Section~\ref{sec:analytical_derivation_Eu}, the input of the recursion - at the continuous level - is the $(k+1)$-points correlation $\corr{0,k}$. 
In the same way - at the discrete level - we start giving a (sparse) discretization of $\corr{0,k}$. It is obtained in two consecutive steps. First, we define the FE approximation of $u^0$ by applying the FE projector $\pi_h$ to $u^0$, i.e., $u_h^0:=\pi_h u^0$. The fully-discrete sparse approximation of $\corr{0,k}$, denoted as $\corrfd{0,k}$, is then obtained by applying the sparse interpolant operator $\spinterp{L,k}$ (with $L$ such that $h_L=h$) to the semi-discrete correlation $\corrsd{0,k}:=u_h^0\otimes\corr{k}$, i.e.,
\begin{equation*}
	\corrfd{0,k}:=\spinterp{L,k}\corrsd{0,k}=\pi_h u^0 \otimes \spinterp{k,L}\corr{k}.
\end{equation*}
Note that the semi-discrete correlation $\corrsd{0,k}$ is an element of $\Cspacemix{n,\gamma}{\bar D^{\times k}}{\mathbb P_\mu(\T_h)}$, whereas the fully-discrete correlation $\corrfd{0,k}$ is an element of the tensor product space $\mathbb P_\mu(\T_h)\otimes \spfe{L,k}$. 

Let $i=k-1,\ldots,0$ fixed. The fully-discrete sparse approximation of the correlation $\corr{k-i,i}$ is obtained as
\begin{equation*}
	\corrfd{k-i,i}:=\spinterp{L,i}\corrsd{k-i,i},
\end{equation*}
where the semi-discrete correlation $\corrsd{k-i,i}$ is defined as the unique solution of the following recursive problem: given all lower order terms $\corrfd{k-i-j,i+j}\in \mathbb P_\mu(\T_h)\otimes \spfe{L,i+j}$ for $j=1,\ldots, k-i$, find $\corrsd{k-i,i}\in \Cspacemix{n,\gamma}{\bar D^{\times i}}{\mathbb P_\mu(\T_h)}$ such that
\begin{align}
	\nonumber
	& \int_D \Grad\corrsd{k-i,i}(x,\vet y)\cdot \Grad \varphi_h(x)\, dx\\
	\label{eq:semidiscrete_corr}
	& \quad =
	-\sum_{j=1}^{k-1} \bc{k-i}{j}
	\int_D (\Trdiag{j+1} \Grad \corrfd{k-i-j,i+j})(x,\vet y)\cdot \Grad \varphi_h(x)\, dx
\end{align}
for all $\varphi\in\mathbb P_\mu(\T_h)$.

In the next theorem we analyze the discretization error.
\begin{theorem}
\label{thm:discretization_recursion}
Let~\eqref{eq:Pmu_approx}, ~\eqref{eq:estimate_Pmu} and~\eqref{eq:approx_Pl} hold. Moreover, let the assumptions of Theorem~\ref{thm:regularity} be satisfied.
Then, it holds
\begin{equation}
	\label{eq:order_convergence}
	\norm{(\corr{k-i,i}-\corrfd{k-i,i})(x,\vet y)}{W^{1,p}_x(D)}=O(\min\{h^\beta,h_L^{s(1-\tau)}\}),
\end{equation}
where $0<\tau<1$ has been introduced in Proposition~\ref{prop:approx_properties}.
\end{theorem}

To prove Theorem~\ref{thm:discretization_recursion}, we need to show some preliminary results.

\begin{lemma}
\label{lem:error_recursion}
Let the assumptions of Theorem~\ref{thm:discretization_recursion} hold, and define 
\begin{equation}
	\label{eq:theta}
	\theta_{n,m}=
	\left\{\begin{array}{ll}
	1, & \text{ if }n=m\\
	0, & \text{ if }n<m\\
	C_S \sum_{j=1}^{n-m}\bc{n}{j}C_{tr}^j\theta_{n-j,m}, &\text{ if }n>m,
	\end{array}\right.
\end{equation}
$C_{tr}$ being as in Proposition~\ref{prop:traccia_Wp}.
Then, it holds:
\begin{align}
	\nonumber
	&\norm{(\corr{k-i,i}-\corrfd{k-i,i})(x,\vet y)}{W^{1,p}_x(D)}\\
	\nonumber
	&\quad\leq
	C_S\, C_{fem}\, h^\beta \sum_{m=i}^{k-1}\, \theta_{k-i,k-m}
	\norm{\corr{k-m,m}(x,\vet y^{(m)};\vet y)}
	{\Cspacemixpedex{0,\gamma}{\vet y^{(m)}}{\bar D^{\times(m-i)}}{W^{2+r,p}_x(D)}}\\
	\nonumber
	&\quad +
	\sum_{m=i}^{k-1} \theta_{k-i,k-m}
	\norm{(\corrsd{k-m,m}-\corrfd{k-m,m})(x,\vet y^{(m)};\vet y)}
	{\Cspacemixpedex{0,\gamma}{\vet y^{(m)}}{\bar D^{\times(m-i)}}{W^{1,p}_x(D)}}\\
	\label{eq:error_recursion}
	&\quad +
	\theta_{k-i,0}
	\norm{(\corr{0,k}-\corrfd{0,k})(x,\vet y^{(k)};\vet y)}
	{\Cspacemixpedex{0,\gamma}{\vet y^{(k)}}{\bar D^{\times(k-i)}}{W^{1,p}_x(D)}}
\end{align}
for all $\vet y:=(y_{k-i+1},\ldots,y_k)\in \bar D^{\times i}$, where $\vet y^{(m)}:=(y_{k-m+1},\ldots,y_{k-i})\in\bar D^{\times(m-i)}$, for $m=i,\ldots,k$.
\end{lemma}

\begin{proof}
Let $k$ be fixed. We prove the result by induction on $i$. 
If $i=k$, then~\eqref{eq:error_recursion} holds as an equality. 

Let $i<k$, and assume, by induction, that~\eqref{eq:error_recursion} holds for all $i+1\leq j\leq k$. Denote $e_{k-i,i}:=\corr{k-i,i}-\corrfd{k-i,i}$, and $f_{k-i,i}:=\corrsd{k-i,i}-\corrfd{k-i,i}$.
By triangular inequality we have:
\begin{equation}
	\label{eq:tr_in}
	\norm{e_{k-i,i}(x,\vet y)}{W^{1,p}_x}
	\leq
	\norm{\corr{k-i,i}-\corrsd{k-i,i}(x,\vet y)}{W^{1,p}_x} + 
	\norm{f_{k-i,i}(\cdot,\vet y)}{W^{1,p}_x}.
\end{equation} 
Using~\cite{Rannacher1982} it is possible to prove that the discrete inf-sup condition for the Laplacian in the spaces $W^{1,p}(D)$-$W^{1,q}(D)$ holds. Applying the Strang's Lemma we have:
\begin{align}
	\nonumber
	& \norm{(\corr{k-i,i}-\corrsd{k-i,i})(x,\vet y)}{W^{1,p}_x}\\
	\label{eq:one}
	& \quad \leq C_S\left(
	\inf_{\varphi_h\in\mathbb P_\mu(\T_h)} \norm{\corr{k-i,i}(x,\vet y)-\varphi_h(x)}{W^{1,p}_x}
	+
	\sup_{\varphi_h\in\mathbb P_\mu(\T_h)} \frac{\abs{\mathcal L_{\vet y}(\varphi_h)-\mathcal L_h(\varphi_h)}}{\norm{\varphi_h}{W^{1,q}_x}}\right),
\end{align}
where $\mcL_{\vet y},\mcL_h:\mathbb P_\mu(\T_h)\rightarrow \R$ are the functionals defining the right-hand side of problems~\eqref{eq:recursion} and~\eqref{eq:semidiscrete_corr}, respectively.
The bound on the first term in the right-hand side of~\eqref{eq:one} follows from the approximation property~\eqref{eq:Pmu_approx}:
\begin{equation}
	\label{eq:two}
	\inf_{\varphi_h\in\mathbb P_\mu(\T_h)} \norm{\corr{k-i,i}(x,\vet y)-\varphi_h(x)}{W^{1,p}_x}
	\leq C_{fem}\,h^\beta \seminorm{\corr{k-i,i}(x,\vet y)}{W^{2+r,p}_x},
\end{equation}
with $\beta=\min\{\mu,2+r\}-1$.
We bound now the second term in the right-hand side of~\eqref{eq:one}. Using the H\"older inequality and Proposition~\ref{prop:traccia_Wp}, we have:
\begin{align}
	\nonumber
	&\abs{\mathcal L_{\vet y}(\varphi_h)-\mathcal L_h(\varphi_h)}\\
	\nonumber
	&\quad \leq \sum_{l=1}^{k-i} \bc{k-i}{l}
	 \abs{\int_D
	 \big(\Trdiag{l+1}\Grad(\corr{k-i-l,i+l}-\corrfd{k-i-l,i+l})\big)
	 (x,\vet y)
	 \cdot\Grad\varphi_h(x)\,dx}\\
	\nonumber
	&\quad\leq \sum_{l=1}^{k-i} \bc{k-i}{l}
	\norm{\big(\Trdiag{l+1}\Grad(\corr{k-i-l,i+l}-\corrfd{k-i-l,i+l})\big)(x,\vet y)}{L^p_x}
	\norm{\Grad \varphi_h}{L^q_x}\\
	\label{eq:three}
	&\quad\leq 
	\sum_{l=1}^{k-i} \bc{k-i}{l} C_{tr}^l
	\norm{e_{k-i-l,i+l}(x,\vet y^{(i+l)};\vet y)}
	{\Cspacemixpedex{0,\gamma}{\vet y^{(i+l)}}{\bar D^{\times l}}{W_x^{1,p}}}
	\norm{\varphi_h}{W^{1,q}_x}.
\end{align}
Inserting~\eqref{eq:one}, \eqref{eq:two} and~\eqref{eq:three} into~\eqref{eq:tr_in}, we have:
\begin{align}
	\nonumber
	& \norm{e_{k-i,i}(x,\vet y)}{W^{1,p}_x}
	\leq 
	C_S\,C_{fem}\, h^\beta \seminorm{\corr{k-i,i}(x,\vet y)}{W^{2+r,p}_x} \\
	\nonumber
	&\quad +
	C_S \sum_{l=1}^{k-i} \bc{k-i}{l} C_{tr}^l
	\norm{e_{k-i-l,i+l}(\cdot,\vet y^{(i+l)};\vet y)}
	{\Cspacemixpedex{0,\gamma}{\vet y^{(i+l)}}{\bar D^{\times l}}{W_x^{1,p}}}\\
	\label{eq:four}
	&\quad + 
	\norm{f_{k-i,i}(x,\vet y)}{W^{1,p}_x}.
\end{align}
Using the inductive assumption on $e_{k-i-l,i+l}$, we get:
\begin{align}
	\nonumber
	& \norm{e_{k-i,i}(x,\vet y)}{W^{1,p}_x}
	\leq C_S\,C_{fem}\, h^\beta \seminorm{\corr{k-i,i}(x,\vet y)}{W^{2+r,p}_x} \\
	\nonumber
	&\quad + C_S \sum_{l=1}^{k-i} \bc{k-i}{l} C_{tr}^l\\
	\nonumber
	&\quad\quad
	\Bigg(
	C_S\, C_{fem}\, h^\beta
	\sum_{m=i+l}^{k-1} \theta_{k-i-l,k-m}
	\norm{E^{k-m,m}(x,\vet y^{(m)};\vet y)}
		{\Cspacemixpedex{0,\gamma}{\vet y^{(m)}}{\bar D^{\times (m-i)}}{W^{2+r,p}_x}}\\
	\nonumber
	&\quad\quad\quad
	\sum_{m=i+l}^{k-1} \theta_{k-i-l,k-m}
	\norm{f_{k-m,m}(x,\vet y^{(m)};\vet y)}
		{\Cspacemixpedex{0,\gamma}{\vet y^{(m)}}{\bar D^{\times (m-i)}}{W^{1,p}_x}}\\
	\nonumber
	&\quad\quad\quad
	\theta_{k-i-l,0}
	\norm{e_0(x,\vet y^{(k)};\vet y)}
		{\Cspacemixpedex{0,\gamma}{\vet y^{(k)}}{\bar D^{\times (k-i)}}{W^{1,p}_x}}
	\Bigg)\\
	\label{eq:five}
	&\quad +
	\norm{f_{k-i,i}(x,\vet y)}{W^{1,p}_x}.
\end{align}
Observe that, by definition of $\theta_{k-i,0}$, we have
\begin{equation}
\label{eq:six}
	C_S \sum_{l=1}^{k-i} \bc{k-i}{l} C_{tr}^l
	\theta_{k-i-l,0}
	=\theta_{k-i,0}.
\end{equation}
Moreover, by switching the sum in $l$ and $m$, and using that $\theta_{k-i,k-i}=1$, we have
\begin{align*}
	C_S \sum_{l=1}^{k-i} \sum_{m=i+l}^{k-1} \bc{k-i}{l} C_{tr}^l\,\theta_{k-i-l,k-m} 
	& = C_S \sum_{m=i+1}^{k-1} \sum_{l=1}^{m-i}
	\bc{k-i}{l} C_{tr}^l \theta_{k-i-l,k-m}\\
	& = \sum_{m=i+1}^{k-1} \theta_{k-i,k-m},
\end{align*}
so that
\begin{align}
	\nonumber
	& C_S \sum_{l=1}^{k-i} \sum_{m=i+l}^{k-1} \bc{k-i}{l} C_{tr}^l\,\theta_{k-i-l,k-m}
	\norm{f_{k-m,m}(x,\vet y^{(m)};\vet y)}{\Cspacemix{0,\gamma}{\bar D^{\times (m-i)}}{W^{1,p}_x}}\\
	\nonumber
	&\quad +
	\norm{f_{k-i,i}(x,\vet y)}{W^{1,p}_x}\\
	\nonumber
	& = 
	\sum_{m=i+1}^{k-1} \theta_{k-i,k-m} 
	\norm{f_{k-m,m}(x,\vet y^{(m)};\vet y)}{\Cspacemix{0,\gamma}{\bar D^{\times (m-i)}}{W^{1,p}_x}}
	+ \norm{f_{k-i,i}(x,\vet y)}{W^{1,p}_x}\\
	\label{eq:seven}
	& = 
	\sum_{m=i}^{k-1} \theta_{k-i,k-m}
	\norm{f_{k-m,m}(x,\vet y^{(m)};\vet y)}{\Cspacemix{0,\gamma}{\bar D^{\times (m-i)}}{W^{1,p}_x}},
\end{align}
and
\begin{align}
	\nonumber
	& C_S^2\,C_{fem}\,h^\beta \sum_{l=1}^{k-i} \sum_{m=i+l}^{k-1} \bc{k-i}{l} C_{tr}^l\,\theta_{k-i-l,k-m}
	\norm{E^{k-m,m}(x,\vet y^{(m)};\vet y)}{\Cspacemix{0,\gamma}{\bar D^{\times (m-i)}}{W^{2+r,p}_x}}\\
	\nonumber
	&\quad +
	C_S\,C_{fem}\, h^\beta \norm{\corr{k-i,i}(x,\vet y)}{W^{2+r,p}_x}\\
	\nonumber
	& =
	C_S\,C_{fem}\,h^\beta \sum_{m+1}^{k-1} \theta_{k-i,k-m}
	\norm{E^{k-m,m}(x,\vet y^{(m)};\vet y)}{\Cspacemix{0,\gamma}{\bar D^{\times (m-i)}}{W^{2+r,p}_x}}\\
	\nonumber
	&\quad +
	C_S\,C_{fem}\, h^\beta \norm{\corr{k-i,i}(x,\vet y)}{W^{2+r,p}_x}\\
	\label{eq:eight}
	& = C_S\,C_{fem}\, h^\beta \sum_{m=i}^{k-1} \theta_{k-i,k-m}
	\norm{E^{k-m,m}(x,\vet y^{(m)};\vet y)}{\Cspacemix{0,\gamma}{\bar D^{\times (m-i)}}{W^{2+r,p}_x}}.
\end{align}
Inserting~\eqref{eq:six}, \eqref{eq:seven} and~\eqref{eq:eight} into~\eqref{eq:five}, we conclude the bound~\eqref{eq:error_recursion}.
\end{proof}

\begin{lemma}
\label{lem:approx_E0k}
Under the assumptions of Theorem~\ref{thm:discretization_recursion} it holds: 
\begin{align}
	\nonumber
	& \norm{\corr{0,k}-\corrfd{0,k}}
	{\Cspacemixpedex{0,\gamma}{y_1,\ldots,y_k}{\bar D^{\times k}}{W^{1,p}_x(D)}}\\
	\label{eq:approx_E0k}
	&\quad \leq 
	\left(C_{\pi_h} h^\beta 
	+C_{\spinterp{L,k}}\, h_L^{s(1-\tau)}
	(C_\pi\,h^\beta+1)\right)
		\norm{\corr{0,k}}
		{\Cspacemixpedex{0,\gamma}{y_1,\ldots,y_k}{\bar D^{\times k}}{W^{2+r,p}_x(D)}}.
\end{align}
\end{lemma}

\begin{proof}
Using the triangular inequality, we have
\begin{align}
	\nonumber
	& \norm{\corr{0,k}-\corrfd{0,k}}
		{\Cspacemixpedex{0,\gamma}{y_1,\ldots,y_k}{\bar D^{\times k}}{W^{1,p}_x(D)}}\\
	\label{eq:E0k_one}
	& \quad \leq 
	\norm{\corr{0,k}-\corrsd{0,k}}
		{\Cspacemixpedex{0,\gamma}{y_1,\ldots,y_k}{\bar D^{\times k}}{W^{1,p}_x(D)}}
	+ \norm{\corrsd{0,k}-\corrfd{0,k}}
	{\Cspacemixpedex{0,\gamma}{y_1,\ldots,y_k}{\bar D^{\times k}}{W^{1,p}_x(D)}}.
\end{align}
We bound the two terms at the right hand side of~\eqref{eq:E0k_one} separately. Using~\eqref{eq:estimate_Pmu}, we have
\begin{align}
	\nonumber
	& \norm{\corr{0,k}-\corrsd{0,k}}
		{\Cspacemixpedex{0,\gamma}{y_1,\ldots,y_k}{\bar D^{\times k}}{W^{1,p}_x(D)}}\\	
	\nonumber
	& \quad 
	= \norm{u^0-\pi_h u^0}{W^{1,p}(D)}
	\norm{\corr{k}}{\Holdermix{0,\gamma}{\bar D^{\times k}}}\\
	\nonumber
	& \quad
	\leq C_{\pi_h} h^\beta \seminorm{u^0}{W^{2+r,p}(D)}
	\norm{\corr{k}}{\Holdermix{0,\gamma}{\bar D^{\times k}}}\\
	\label{eq:E0k_two}
	& \quad =
	C_{\pi_h} h^\beta
	\norm{\corr{0,k}}{\Cspacemix{0,\gamma}{\bar D^{\times k}}{W^{2+r,p}(D)}}.
\end{align}
Moreover, applying Proposition~\ref{prop:approx_properties}, the triangular inequality, and~\eqref{eq:estimate_Pmu}, we have
\begin{align}
	\nonumber
	& \norm{\corrsd{0,k}-\corrfd{0,k}}
	{\Cspacemixpedex{0,\gamma}{y_1,\ldots,y_k}{\bar D^{\times k}}{W^{1,p}_x(D)}}
	\leq C_{\spinterp{L,k}}\, h_L^{s(1-\tau)}
	\norm{\corrsd{0,k}}{\Cspacemixpedex{n,\gamma}{y_1,\ldots,y_k}{\bar D^{\times k}}{W^{1,p}_x(D)}}\\
	\nonumber
	&\quad\leq
	C_{\spinterp{L,k}}\, h_L^{s(1-\tau)}
	\left(\norm{\corr{0,k}-\corrsd{0,k}}{\Cspacemixpedex{n,\gamma}{y_1,\ldots,y_k}{\bar D^{\times k}}{W^{1,p}_x(D)}}
	+\norm{\corr{0,k}}{\Cspacemixpedex{n,\gamma}{y_1,\ldots,y_k}{\bar D^{\times k}}{W^{1,p}_x(D)}}\right)\\
	\label{eq:E0k_three}
	&\quad\leq
	C_{\spinterp{L,k}}\, h_L^{s(1-\tau)}
	(C_\pi\,h^\beta+1)\norm{\corr{0,k}}{\Cspacemixpedex{n,\gamma}{y_1,\ldots,y_k}{\bar D^{\times k}}{W^{1,p}_x(D)}}.
\end{align}
The result is then proved inserting~\eqref{eq:E0k_two} and~\eqref{eq:E0k_three} into~\eqref{eq:E0k_one}.
\end{proof}

\begin{proof}[Theorem~\ref{thm:discretization_recursion}]
To prove~\eqref{eq:order_convergence} we bound each term at the right-hand side of~\eqref{eq:error_recursion}, separately.
Applying Proposition~\ref{prop:bound_recursion}, we have:
\begin{align}
	\nonumber
	& C_S\, C_{fem}\, h^\beta \sum_{m=i}^{k-1}\, \theta_{k-i,k-m}
	\norm{\corr{k-m,m}(x,\vet y^{(m)};\vet y)}
	{\Cspacemixpedex{0,\gamma}{\vet y^{(m)}}{\bar D^{\times(m-i)}}{W^{2+r,p}_x(D)}}\\
	\label{eq:part1}
	& \quad \leq
	C_S\, C_{fem}\, h^\beta 
	\left(\sum_{m=i}^{k-1}\, \theta_{k-i,k-m}\,\lambda_{k-m}\right)
	\norm{\corr{0,k}(x,\vet y^{(k)};\vet y)}
		{\Cspacemixpedex{0,\gamma}{\vet y^{(k)}}{\bar D^{\times(k-i)}}{W^{2+r,p}_x(D)}}.
\end{align}
Applying Proposition~\ref{prop:approx_properties}, the triangular inequality, and Proposition~\ref{prop:bound_recursion}, we have:
\begin{align}
	\nonumber
	& \sum_{m=i}^{k-1} \theta_{k-i,k-m}
	\norm{(\corrsd{k-m,m}-\corrfd{k-m,m})(x,\vet y^{(m)};\vet y)}
	{\Cspacemixpedex{0,\gamma}{\vet y^{(m)}}{\bar D^{\times(m-i)}}{W^{1,p}_x(D)}}\\
	\nonumber
	&\quad\leq
	\sum_{m=i}^{k-1} \theta_{k-i,k-m} C_{\spinterp{L,m}}\, h_L^{s(1-\tau)}
	\norm{\corrsd{k-m,m}(x,\vet y^{(m)};\vet y)}
	{\Cspacemixpedex{n,\gamma}{\vet y^{(m)}}{\bar D^{\times(m-i)}}{W^{1,p}_x(D)}}\\
	\nonumber
	&\quad\leq h_L^{s(1-\tau)}
	\sum_{m=i}^{k-1} \theta_{k-i,k-m} C_{\spinterp{L,m}}
	\Big(\norm{\corr{k-m,m}(x,\vet y^{(m)};\vet y)}
	{\Cspacemixpedex{n,\gamma}{\vet y^{(m)}}{\bar D^{\times(m-i)}}{W^{1,p}_x(D)}}\\
	\nonumber
	&\quad\quad +
	\norm{(\corr{k-m,m}-\corrsd{k-m,m})(x,\vet y^{(m)};\vet y)}
	{\Cspacemixpedex{n,\gamma}{\vet y^{(m)}}{\bar D^{\times(m-i)}}{W^{1,p}_x(D)}}\Big)\\
	\nonumber
	&\quad\leq h_L^{s(1-\tau)}
	\sum_{m=i}^{k-1} \theta_{k-i,k-m} C_{\spinterp{L,m}}
	\Big(\lambda_{k-m}
	\norm{\corr{0,k}(x,\vet y^{(k)};\vet y)}
	{\Cspacemixpedex{n,\gamma}{\vet y^{(k)}}{\bar D^{\times(k-i)}}{W^{2+r,p}_x(D)}}\\
	\nonumber
	&\quad\quad +
	C_{\pi_h}\,h^\beta
		\norm{\corr{k-m,m}(x,\vet y^{(m)};\vet y)}
		{\Cspacemixpedex{n,\gamma}{\vet y^{(m)}}{\bar D^{\times(m-i)}}{W^{2+r,p}_x(D)}}\Big)\\
	\label{eq:part2}
	&\quad\leq h_L^{s(1-\tau)}(C_{\pi_h}\,h^\beta+1)
	\sum_{m=i}^{k-1} \theta_{k-i,k-m} C_{\spinterp{L,m}}
	\lambda_{k-m}
	\norm{\corr{0,k}(x,\vet y^{(k)};\vet y)}
	{\Cspacemixpedex{n,\gamma}{\vet y^{(k)}}{\bar D^{\times(k-i)}}{W^{2+r,p}_x(D)}}.
\end{align}
The result follows by applying Lemma~\ref{lem:approx_E0k}, and inserting~\eqref{eq:part1}, and~\eqref{eq:part2}  into~\eqref{eq:error_recursion}.
\end{proof}

\begin{remark}
The finite dimensional spaces $P_\mu(\T)$ and $V_\ell$ are defined on the same physical domain $D$. It is then natural to take $V_\ell=P_\mu(\T_\ell)$ - $\T_\ell$ having  discretization parameter $h_\ell$ - and $h=h_L$. Then,~\eqref{eq:order_convergence} becomes:
$$\norm{(\corr{k-i,i}-\corrfd{k-i,i})(x,\vet y)}{W^{1,p}_x(D)}=O(h^{\min\{\beta,s(1-\tau)\}}).$$
\end{remark}

%%%%%%%%%%%%%%%%%%%%%%%%%%%%%%%%%%%%%%%%%%%%%%%%%%%%%%%%%%%%%%%%%%%%%%
\section{Conclusions}
\label{sec:conclusions}
%%%%%%%%%%%%%%%%%%%%%%%%%%%%%%%%%%%%%%%%%%%%%%%%%%%%%%%%%%%%%%%%%%%%%%

This paper addresses the computation of an approximation for the expected value of the unique stochastic solution $u$ to the Darcy problem with lognormal permeability coefficient. In particular, we adopt the perturbation method - approximating the solution by its Taylor polynomial $T^Ku$ - in combination with the moment equation technique - approximating $\mean{u}$ by $\mean{T^Ku}$.
The first moment equation is recalled, and its recursive structure is explained. In particular, for each $k=0,\ldots,K$, a recursion on the $(i+1)$-points correlation $\corr{u^{k-i}\otimes Y^{\otimes i}}$, $i=1,\ldots,k$, is needed.
Well-posedness and regularity results for the recursion satisfied by $\corr{k-i,i}$ are proved. In particular we show that $\corr{u^{k-i}\otimes Y^{\otimes i}}\in\Cspacemix{n,\gamma}{\bar D^{\times i}}{W^{2+r,p}(D)}$, under the assumptions $Y\in\Holder{n,\gamma}{\bar D}$ a.s. and $u^0\in W^{2+r,p}(D)\cap W^{1,p}_0(D)$.
Finally, a sparse discretization for the recursion is analyzed, and the convergence of the sparse discretization error is proved.

The procedure proposed in this paper can be used also to approximate higher moments of $u$. In particular, we refer to~\cite{Bonizzoni2016} for the recursion on the two-points correlation of $u$, $\mean{u\otimes u}$.
Moreover, the bounds on sparse grid approximations derived in this work could also be useful to establish convergence estimates for low rank approximations as the Tensor Train considered in~\cite{Bonizzoni2016} (see, e.g.,~\cite{Schneider2014}).

\end{document}